\documentclass[10pt,a4paper]{amsart}
\usepackage{tikz,tkz-graph}
\usetikzlibrary{arrows}
\usetikzlibrary{calc}
\usepackage{pgfplots}
\pgfplotsset{every axis/.append style={
                    axis x line=middle,    
                    axis y line=middle,    
                    axis line style={-,color=blue}, 
                    xlabel={$x$},          
                    ylabel={$y$},          
            }}
            
\usepackage[latin1]{inputenc}
\usepackage{amsmath}
\usepackage{amsthm}
\usepackage{amsfonts}
\usepackage{amssymb}
\usepackage{enumerate}
\usepackage{graphicx}
\usepackage{hyperref}
\usepackage{nicefrac}
\usepackage{cancel}

\DeclareMathOperator{\orb}{orb}

\DeclareMathOperator{\Sing}{Sing}

\DeclareMathOperator{\Cl}{Cl}

\DeclareMathOperator{\Opt}{Opt}

\def\qa{{\boldsymbol{{G}}}}

\def\ZZ{\mathbb{Z}}

\def\NN{\mathbb{N}}
\def\CC{\mathbb{C}}
\def\LL{\mathbb{L}}
\def\Q{\mathbb{Q}}
\def\QQ{\mathbb{Q}}

\def\PP{\mathbb{P}}
\def\K{\kappa}
\def\w{\omega}

\def\cM{\mathcal{M}}
\def\cO{\mathcal{O}}

\def\logres{{\rm LR}}

\def\O{\mathcal{O}}
\def\R{\Delta}
\def\ii{r}
\def\j{s}
\def\gG{\Gamma}
\def\G{\rm G}
\def\greedy{greedy\,}

\def\leftblowup#1#2{\smash{\mathop{\longleftarrow}
\limits^{\array{ll} _{#1} \\ ^{#2} \endarray}}}

\def\longleftmap#1{\smash{\mathop{\longleftarrow}\limits^{#1}}}

\newcommand\nul{\text{nul}}

\newtheorem{thm}{Theorem}[section]  
\newtheorem{prop}{Proposition}[section]
\newtheorem{proper}{Properties}[section]%
\newtheorem{problem}{Problem}[section]%
\newtheorem{cor}{Corollary}[section]
\newtheorem{lemma}{Lemma}[section]
\newtheorem{propdef}{Proposition-Definition}[section]
\theoremstyle{remark}
\newtheorem{rem}{Remark}[section]

\theoremstyle{definition}
\newtheorem{dfn}{Definition}[section]
\newtheorem{exam}{Example}[section]

\makeatletter
\let\c@lemma\c@thm
\let\c@prop\c@thm
\let\c@propdef\c@thm
\let\c@proper\c@thm
\let\c@problem\c@thm
\let\c@conj\c@thm
\let\c@cor\c@thm
\let\c@rem\c@thm
\let\c@dfn\c@thm
\let\c@notation\c@thm
\let\c@exam\c@thm
\makeatother
\title
[The space of curvettes of quotient singularities...]
{The space of curvettes of quotient singularities and associated invariants}

\author[J.I. Cogolludo]{Jos{\'e} Ignacio Cogolludo-Agust{\'i}n}
\address{Departamento de Matem\'aticas, IUMA\\ 
Universidad de Zaragoza\\ 
C.~Pedro Cerbuna 12\\ 
50009 Zaragoza, Spain} 
\email{jicogo@unizar.es} 

\author[J. Martín]{Jorge Martín-Morales}
\address{Centro Universitario de la Defensa-IUMA\\ 
Academia General Militar\\ 
Ctra.~de Huesca s/n.\\ 
50090, Zaragoza, Spain} 
\email{jorge@unizar.es}

\begin{document}

\thanks{Both authors are partially supported by
the Spanish Government MTM2013-45710-C2-1-P and 
\emph{E15 Grupo Consolidado Geometr\'{\i}a} from the Gobierno de Aragón. 
The second author is also supported by FQM-333, from  
Junta de Andaluc{\'\i}a.}

\subjclass[2010]{32S05, 14H50, 32S25, 14F45}  

\begin{abstract}
This paper deals with a complete invariant $\R_X$ for cyclic quotient surface singularities.
This invariant appears in the Riemann Roch and Numerical Adjunction Formulas for normal surface singularities.
Our goal is to give an explicit formula for $\R_X$ based on the numerical information of $X$, that is,
$d$ and $q$ as in $X=X(d;1,q)$.
In the process, the space of curvettes and generic curves is explicitly described.
We also define and describe other invariants of curves in $X$ such as the LR-logarithmic eigenmodules, 
$\delta$-invariants, and their Milnor and Newton numbers. 
\end{abstract}

\maketitle

\section*{Introduction}
For a projective normal surface $X$ the following generalized Riemann Roch formula can be deduced 
(see e.g.~\cite{Corti,MR927963,MR0450628,Blache-RiemannRoch})
$$
\chi(\cO_X(D))=\chi(X)+\frac{1}{2}D\cdot (D-K_X)+R_X(D),
$$
where $R_X:\Cl(X)\to \QQ$ is a map defined on the $\QQ$-divisor class of~$X$, that is, on the group of Weil divisors up 
to Cartier. The invariant $R_X$ is in fact defined locally, that is, $R_X(D)=\sum_{x\in \Sing(X)} R_{X,x}(D)$. 
In~\cite{Blache-RiemannRoch} also formulas for $R_{X,x}(nK_X)$ are given for $K_X$ the canonical divisor. 
Such formulas depend on the discrepancy of $X$ and the fractional part of the pluricanonical divisor~$nK_X$.

In a related context, let $X=\PP^2(w_0,w_1,w_2)$ be a weighted projective plane and $D=\{f=0\}$ a quasi-projective curve 
of degree~$k$ in~$X$. In~\cite{CMO14} the following Numerical Adjunction Formula was proven
\begin{equation}
\label{eq-adjformula}
h^0(X;\cO_X(k-\sum w_i))=g_{\w,k}-\sum_{x\in \Sing(X)} \R_{X,x}(k),
\end{equation}
where $g_{\w,k}=\frac{k(k-\sum w_i)}{2w_0w_1w_2}+1$ and $\R_{X,x}(k)$ is an invariant depending only on the cyclic 
quotient surface singularity~$(X,x)$ and the $\QQ$-divisor class $\xi^k\in \qa_d=\Cl(X)$ ($\xi$ a $d$-th root of unity)
of a divisor in~$X=\CC^2/\qa_d$ a cyclic quotient singularity of order~$d$.
Since $R_{X,x}(D)=-\R_{X,x}(-D)$, this provides an interpretation for the Adjunction Formulas showed 
in~\cite{Blache-RiemannRoch} for general projective normal surfaces.

The purpose of this paper is to describe the invariant $\R_X:\Cl(X)\to \QQ$ of any cyclic quotient surface 
singularity~$X$. In order to do so we consider it as the difference of two invariants of germs: a $\delta$-invariant 
(see~\cite{CMO12}) and a $\K$-invariant (introduced in~\cite{Ortigas13PhD,Ortigas-cr}). Calculations of $\R_X$
can be effectively carried out for generic curves. Since $\R_X$ does not depend on the representative chosen in 
the divisor class, the calculation of a particular case provides an effective formula for $\R_X$. 
This is why we are interested in describing the space of curvettes and other generic $\QQ$-divisors on~$X$.

The main results in this paper can be summarized as follows. 
Let $X:=X(d;1,q)$ be a cyclic quotient surface singularity and fix $0\leq k<d$.
First, the concept of \emph{generic} germ in a given divisor class is defined as a minimal element in the multivaluation 
given by a minimal resolution of the singularity (see section~\ref{sec-curvettes} for details).
Our first goal will be to describe a generic element of degree $k$.
Consider $\mathbf q=[q_1,\dots,q_n]$ the Hirzebruch-Jung decomposition of $\frac{d}{q}$. 
In Section~\ref{sec-arithmetics} a list of integers $[k]=[k_0,k_1,\dots,k_n,k_{n+1}]$, referred to as the 
\greedy $X$-decomposition of $k$, is defined. The following description of generic $\QQ$-divisors in $X$ is given.

\begin{thm}
\label{thm-main-generic}
Let $[k]=[k_0,\dots,k_{n+1}]$ be the \greedy $X$-decomposition of $0\leq k<d$ and consider the germ 
\begin{equation}
\label{eq-fgeneric}
f=\prod_{i=1}^{n} \prod_{j=1}^{k_i} (x^{q_i}-\lambda^i_j y^{\bar q_i})\in \cO_X(k),
\end{equation}
with $\lambda^i_j\in \CC^*$ and $\lambda^i_{j_1}\neq \lambda^i_{j_2}$, $j_1\neq j_2$.
Then $f$ is generic.

Moreover, any generic germ $g\in \cO_X(k)$ is such that $\Gamma_\LL(g)=\Gamma_\LL(f)$ for an $f$ as above.
\end{thm}

This allows one to give a description of~$\R_X$.

\begin{thm}
\label{thm-Delta}
Let $f$ be a generic curve of degree~$k\neq 0$ in $X$. Then
$$
\R_X(k)=\delta_X(f)-\K_X(f),
$$
where 
$\K_X(f)=\|k\|_1-1$ and $\delta_X(f)$ can be obtained recursively as
$$\delta_X(f)=\frac{k(k-1-q+d)}{2dq}+\delta_{X_2}(\tilde f),$$
where $X_2=X(q_1;1,q_2)$ and $\tilde f$ is the strict transform of $f$ via the $(1,q)$-weighted blow-up of~$X$.

This describes effectively the map
$$
\array{rrcl}
\R_X: &\Cl(X)\cong \ZZ_d& \longrightarrow & \QQ\\
& k & \mapsto & \delta_X(f)-\K_X(f),
\endarray
$$
in terms of $k$ and $\mathbf q$.
\end{thm}

We also give a structure theorem on the $\cO_X$-eigenmodule of quasi-invariant germs $\cO_X(k)$ as follows.

\begin{thm}
\label{thm-descomp}
If $k\in \ZZ$ and $[k]=[k_0,k_1,\dots,k_n,k_{n+1}]$ is the \greedy $X$-de\-com\-po\-si\-tion of $k$, then
$$\cO_X(k)=\bigoplus_{i=1}^n \cO_X(q_i)^{k_i}.$$
\end{thm}


In Section~\ref{sec-LRlogarithmic} the module of LR-logarithmic forms $\cM^{\nul}_f$ is considered as a 
tool to determine the $\K$-invariant of a germ~$f$. This module is briefly defined as the submodule of 
2-forms whose pull-back after resolution can be extended holomorphically over the exceptional divisors.
In the following theorem $\cM^{\nul}_f$ is described.

\begin{thm}
\label{thm-Mnul}
Let $f\in \cO_X(k)$ be a generic germ where $[k]=[k_0,k_1,\dots,k_n,k_{n+1}]$ is the \greedy $X$-decomposition 
of $k$, then
$$\cM^{\nul}_f = \cO_X(k)\otimes \cO_X(w)=\bigoplus \cO_X(q_i)^{k_i+c_i-2}.$$
\end{thm}

The left-hand-side equality is the main result of Theorem~\ref{thm-Mnul} whereas the right-hand-side equality 
is a direct consequence of Theorem~\ref{thm-descomp}.

In the process of proving the main results we define the Milnor number of a germ and the Newton number of a 
polygon for cyclic quotient surface singularities and extend Kouchnirenko's Theorem (see Theorem~\ref{thm-mu}) 
in this context and give an effective formula in Proposition~\ref{prop-muk}. Note that this Milnor number 
differs from the one defined in~\cite{ABLM-milnor-number}.

The paper is organized as follows: in Section~\ref{sec-settings} we give the necessary definitions and notation
to set the concepts to be dealt with in this paper such as the space of germs, curvettes, generic curves,
$\delta$-invariant, $\K$-invariant, $\R_X$, Milnor number, Newton polygons, and Newton numbers.
In Section~\ref{sec-arithmetics} we develop the arithmetical properties such as the $X$-\greedy decomposition of an
integer, necessary to prove Theorems~\ref{thm-main-generic}, \ref{thm-Delta}, and \ref{thm-descomp}. 
Section~\ref{sec-LRlogarithmic} is devoted to analyzing LR-logarithmic eigenmodules
for generic germs in order to prove their structure Theorem~\ref{thm-Mnul}.

\section{Settings and Definitions}
\label{sec-settings}
Let us recall some definitions and properties on quotient surfaces, embedded $\Q$-resolutions, and weighted blow-ups, 
see~\cite{AMO11a,Dolgachev82,fulton-intersection} for a more detailed exposition.

\subsection{Quotient surface singularities}\label{sec-quotient}
Let $\qa_{d}$ be the cyclic group of $d$-th roots of unity generated by a root of unity~$\xi$. 
Consider a vector of weights $(a,b)\in\ZZ^2$ and the action
\begin{eqnarray*}
\qa_d \times \CC^2 & \overset{}{\longrightarrow} & \CC^2, \\
(\xi, (x,y)) & \mapsto & (\xi^{a}\, x,\xi^{b}\, y).
\end{eqnarray*}
The set of all orbits $\CC^2 / \qa_{d}$ is called a {\em cyclic quotient space of type $(d;a,b)$} and it is denoted by $X(d;a,b)$.
After changing the corresponding primitive $d$-th root of unity and transforming the action
into a small one, i.e.~$\gcd(d,a)=\gcd(d,b)=1$, the quotient space can always be assumed
to be of the form $(d;1,q)$ with $\gcd(d,q)=1$.

\subsection{Embedded \texorpdfstring{$\Q$}{Q}-resolutions and weighted blow-ups}
\label{subsec-wblowup}
An {\em embedded $\Q$-reso\-lu\-tion} of a $\QQ$-divisor $\{f=0\} \subset X(d;a,b)$
is a proper analytic map $\pi: Y \to X(d;a,b)$ such that:
\begin{enumerate}
\item $Y$ is an orbifold having abelian (cyclic) quotient singularities,
\item $\pi$ is an isomorphism over $Y\setminus \pi^{-1}(0)$,
\item $\pi^{-1}(f)$ is a $\Q$-normal crossing divisor on $Y$ (see
\cite[Definition~1.16]{Steenbrink77}).
\end{enumerate}

As a key tool to construct embedded $\Q$-resolutions we will recall toric transformations
or weighted blow-ups in this context (see~\cite{Oka-nondegenerate} as a general reference),
which can be interpreted as blow-ups of $\mathfrak{m}$-primary ideals.

The $(p,q)$-weighted blow-up is a birational morphism $\pi : \widehat{X(d;a,b)} \to X(d;a,b)$ that can be described by 
covering $\widehat{X(d;a,b)}$ with two charts $\widehat{U}_1$ and $\widehat{U}_2$.
For instance $\widehat{U}_1$ is of type $X \Big( \displaystyle\frac{pd}{e}; 1, \frac{-q+a' pb}{e} \Big)$, 
with $a'a=b'b\equiv 1 \mod (d)$ and $e = \gcd(d,pb-qa)$, and the equations are given by 
\begin{equation}
\label{eq-charts}
\begin{array}{rcl}
X \left( \displaystyle\frac{pd}{e}; 1, \frac{-q+a' pb}{e} \right)  & \longrightarrow &
\widehat{U}_1, \\ 
\big[ (x^e,y) \big] & \mapsto &  
\big[ ((x^p,x^q y),[1:y]_{\w}) \big]_{(d;a,b)}
\end{array}
\end{equation}
In particular, if the determinant $\left| \begin{smallmatrix} a & b \\ p & q \end{smallmatrix} \right| = 0$,
then $e=d$ and $\widehat{U}_1 = X(p;-d,q)$. The discussion for the second chart is analogous.

The exceptional divisor $E = \pi^{-1}(0)$ is identified with 
$\PP^1_{(p,q)}/\qa_{d}$. The singular points are cyclic 
and correspond to the origins of the charts.

Any cyclic quotient surface singularity $X$ can be described as $X:=X(d;1,q)$. This notation is not 
canonical since $X(d;1,q)=X(d;1,q')$, where $q' q=1 \mod d$. 
The classical well-known resolution of the surface $X:=X(d;1,q)$ is the so-called
\emph{Hirzebruch-Jung resolution} and it is very related to the Hirzebruch-Jung continued fraction
of $\frac{d}{q}$. To fix the notation, let us briefly recall it.

Let $q_0 := d$ and $q_1 := q$, and denote by $q_2, \ldots, q_n \in \mathbb{N}$ such that
$q_{i-1} = c_i q_i - q_{i+1}$, $i \geq 1$, with $q_1 > q_2 > \cdots > q_n := 1 > q_{n+1}:=0$.
The \emph{Hirzebruch-Jung continued fraction} of $\frac{d}{q}$ is $[c_1,\ldots,c_n]$,
where
$$
\frac{d}{q}=c_1-\frac{1}{c_2-\frac{1}{c_3-...}}
$$
and $c_{i+1}:=\left\lceil \frac{q_i}{q_{i+1}}\right\rceil$ is the round-up of the fraction $\frac{q_i}{q_{i+1}}$.
These numerical data encode all the necessary information of the resolution of $X$ as follows.

Consider the $(1,q)$-weighted blow-up at the origin of $X$. One obtains an exceptional
divisor $E_1$ with self-intersection number $-c_1^2$. If $q=1$ the new ambient space is
smooth and the resolution process is over. If $q >1$, then $E_1$ contains a singular point
of type $(q;1,-d)$ which is equal to $(q;1,q_2)$ since $-d \equiv q_2 \mod q$.
Repeating the same procedure until the final surface is smooth, one eventually
obtains $n$ exceptional divisors $E_1,\ldots,E_n$, all of them isomorphic to $\mathbb{P}^1$,
with self-intersection number $-c_i^2$ giving rise to a bamboo-shaped graph.


\subsection{Spaces of germs}

Consider $X=X(d;a,b)$ and $\rho:(\CC^2,0)\to X$ the projection defined over the quotient surface by the cyclic action of 
order $d$ on $(\CC^2,0)$. The local ring $\cO_{\CC^2,0}$ of functions on $(\CC^2,0)$, admits a \emph{cyclic} 
graduation given by quasi-invariants
\begin{equation}
\label{eq-graduation}
\cO_{\CC^2,0}=\bigoplus_{k=0}^{d-1} \cO_X(k),
\end{equation}
where $\cO_X(k)=\{f\in \cO_{\CC^2,0}\mid f(\xi\cdot(x,y))=\xi^kf(x,y), \forall \xi \in \qa_{d} \}$.
The notation $\cO_X(k)$ is justified since its elements, however do not define functions on $X$, they
determine a well-defined set of zeroes in $X$, that is, $\{(x,y)\in \CC^2\mid f(x,y)=0\}$ is $\rho$-saturated and hence
it defines a Weil divisor in~$X$. This explains why $\cO_X(k)$ is also called the \emph{eigenmodule} associated with~$k$
(c.f.~\cite{Reid-Surface}). 
More precisely the space of eigenfunctions of the morphism $\xi\cdot:\cO_{\CC^2,0}\to \cO_{\CC^2,0}$ defined by 
$f(x,y)\mapsto f(\xi\cdot (x,y))$ with eigenvalue~$\xi^k$. These eigenmodules are in one-to-one correspondence with 
the (isomorphism classes of) divisorial submodules on $X$, that is, the group of Weil divisor classes is naturally 
isomorphic to $\qa_d$ and each class is given by the elements in~$\cO_X(k)$. Note that:

\begin{proper}
\label{property}
\mbox{}
\begin{enumerate}
\item\label{property1} $\cO_X(0)=\cO_X$ is the ring of functions on $X$,
\item\label{property2} $\cO_X(k_1)=\cO_X(k_2)$ whenever $k_1\equiv k_2 \mod d$,
\item\label{property3} $\cO_X(k_1)\otimes \cO_X(k_2)\subset \cO_X(k_1+k_2)$,
\item\label{property4} $\cO_X(k)$ is a f.g. monomial $\cO_X$-module.
\end{enumerate}
\end{proper}

\begin{proof}[Proof of~\emph{(4)}]
Properties~\eqref{property1} and \eqref{property3} imply that $\cO_X(k)$ is an $\cO_X$-module. 
Also, note that any germ in $\cO_X(k)$, say $f(x,y)=\sum_{\ii,\j} a_{\ii,\j}x^\ii y^\j$ satisfies that
$f(\xi^a x,\xi^b y)=\xi^k f(x,y)$, where $\xi$ is a primitive $d$-th root of unity. 
Hence $\xi^{a\ii+b\j} = \xi^k$ for all $i,j$ such that $a_{\ii,\j}\neq 0$, that is, $a\ii+b\j\equiv k \mod d$ 
and hence, each non-trivial monomial of $f(x,y)$ is in $\cO_X(k)$.
Moreover, this module is generated by the finite set of monomials 
$\{x^\ii y^\j \mid \ii,\j\in \{0,...,d-1\}, \ a\ii+b\j\equiv k \! \mod d \}$. 
\end{proof}

%
%
%
%
%

\subsection{LR-Logarithmic eigenmodules}
\label{sec-logarithmic}
%
%
Let $\{f=0\}$ be a germ in $X=X(d;a,b)$ with $f \in \O_X(k)$ and consider $D=(f)$ its associated Weil divisor.
The $\cO_{\CC^2}$-modules of differential forms on $X\setminus D$ also inherit a cyclic graduation based on their 
eigenmodules similar to that of~\eqref{eq-graduation}. Multiplication by $\frac{dx\wedge dy}{f}$ induces a morphism
$$
\array{rrcl}
\varphi:&\cO_X(s)& \longrightarrow & \Omega^2_{X}[D](s-k-w)\\
&h& \mapsto & h \frac{dx\wedge dy}{f}.
\endarray
$$
where $\w=d-a-b$ is the degree of the canonical divisor on $X$.
Let us now fix a $\Q$-resolution $\pi: Y \to X$ of $D$. The notion of \emph{log-resolution logarithmic eigenmodule}
(LR for short) is defined in~\cite{Ortigas-cr,CMO14} as 
$\Omega^i_{X}(\logres \langle D\rangle)=\pi_*\Omega^i_{Y}(\log \langle\pi^*D\rangle)$. 
Denote by $\mathcal{M}^{\nul}_{f,\pi}(s)\subset \O_X(s)$ the $\O_X$-eigenmodule resulting as a pull-back of 
$\Omega^2_{X}(\logres \langle D\rangle)(s-k-w)$, namely consisting of all $h \in \O_X(s)$ such that the $2$-form
$$
h \frac{dx \wedge dy}{f} \in \Omega^2_X[D](s-k-w)
$$
is LR-logarithmic, with respect to $f$ and $\pi$, and admits a holomorphic extension outside the strict transform of 
$f$ under $\pi$. Note that the quotient 
$\O_X(s) / \mathcal{M}^{\nul}_{f,\pi}(s)$ has the structure of a finite dimensional complex vector space as long as 
$f$ defines an isolated singularity.

The following integer number does not depend on the chosen resolution:
\begin{equation}
 \label{eq-k0}
\K_{X}(f):=\dim_{\CC} \frac{\O_X(k+w)}{{\mathcal M}^{\nul}_{f}(k+w)}.
\end{equation}
For instance, it is known (see \cite[Chapter $2$]{JIphd}) that if $X=(\CC^2,0)$ and $f$ is a holomorphic germ, 
then $\K_{X}(f)$ is the $\delta$-invariant of the singularity.

\begin{prop} \label{prop-KP-blowup}
Let $\pi: \widehat{X} \to X$ be the weighted $(p,q)$-blow-up defined in Section~\ref{subsec-wblowup} and
consider $f \in \O_X(k)$. Then,
\begin{equation} \label{eq-KP}
\K_{X}(f) = \K_\pi + \sum_{P \in E \cap V(\widehat{f})} \K_P(\widehat{f}),
\end{equation}
where $\K_\pi = \# \{ (i,j) \in \mathbb{Z}^2 \mid i,j \geq 1, \ p i + q j \leq \nu_{p,q}(f),
\ ai+bj \equiv k \mod d \}$ and $\nu_{p,q}(f)$ denotes the multiplicity of $f(x^p,y^q)$.
\end{prop}

\begin{proof}
Consider $h \in \O_X(k+w)$ and the $2$-form $\psi = h \displaystyle \frac{dx \wedge dy}{f}$ and let us
calculate the pull-back of $\psi$ after the blowing-up $\pi$,
\begin{equation}
\label{eq-2form}
\psi \ \longleftmap{\pi} \ \frac{p}{e} \, x^N \ \widehat{h} \, \frac{dx \wedge dy}{\widehat f},
\end{equation}
where $N = (\nu_{p,q}(h)-\nu_{p,q}(f)+p+q-e)/e$ and $e=\gcd(d,pb-qa)$, see Section~\ref{subsec-wblowup}.
Thus $h \in \cM_f^\nul(k+w)$ iff $N \geq 0$ and
$\widehat{h} \in \cM_{\widehat{f},P}^\nul$ for all $P \in E \cap V(\widehat{f})$. This proves
$$
\K_{X}(f) = \tilde{\K}_\pi + \sum_{P \in E \cap V(\widehat{f})} \K_P(\widehat{f}),
$$
where $\tilde{\K}_\pi = \dim_{\CC} \left( \frac{\O_X(k+w)}{\{ h \, \mid \,\nu_{p,q}(h) \, \geq \, 
\nu_{p,q}(f) - p - q + e\}} \right)$. 

It remains to show that $\tilde{\K}_\pi = \K_\pi$. Since both $\cO_X$-modules are monomial, the dimension of 
the quotient can be computed simply by counting the monomials in $\O_X(k+w)$ not in 
$\{ h \in \O_X(k+w) \, \mid \,\nu_{p,q}(h) \, \geq \, \nu_{p,q}(f) - p - q + e\}$.
Identifying each monomial $x^iy^j$ with the integral point $(i,j)$ in $\mathbb{Z}_{\geq 0}^2$, one obtains
\begin{align*}
\tilde{\K}_{\pi} &= \# \{ (i,j) \mid i,j \geq 0, \ pi+qj < \nu_{p,q}(f) - p - q + e, \ ai+bj \equiv k - a - b \!\! \mod d \} \\
&= \# \{ (i,j) \mid i,j \geq 1, \ pi+qj < \nu_{p,q}(f) + e, \ ai+bj \equiv k \!\! \mod d \} \\
&= \# \{ (i,j) \mid i,j \geq 1, \ pi+qj \leq \nu_{p,q}(f) , \ ai+bj \equiv k \!\! \mod d \}.
\end{align*}
\end{proof}

The latter equality is a direct consequence of the following result.

\begin{lemma}\label{lemma-ij}
Under the conditions above 
$$
\array{c}
\{(i,j)\mid pi+qj \leq \nu_{p,q}(f),\ ai+bj \equiv k \!\! \mod d\}=\\
\{(i,j)\mid \nu_{p,q}(f) \geq pi+qj \equiv \nu_{p,q}(f) \!\! \mod e,\ ai+bj \equiv k \!\! \mod d\}.
\endarray$$
\end{lemma}

\begin{proof}
Consider the system $\{ p i + q j + \ell = \nu_{p,q}(f), ai + bj \equiv k \! \mod d \}$. It will
be shown that $e$ divides $\ell$. Since $0 \neq f \in \O_X(k)$, there exist $0 \neq (i_0, j_0) \in \NN^2$ such that
$\nu_{p,q}(f) = pi_0 + qj_0$ with $k=ai_0 + bj_0$. Then,
$$
\begin{cases}
p(i-i_0) + q(j-j_0) \equiv - \ell, \\
a(i-i_0) + b(j-j_0) \equiv 0.
\end{cases}
$$
Multiplying the first equation by $a$ and the second one by $p$, one obtains $(aq-pb)(j-j_0) \equiv - a \ell$,
thus $e:=\gcd(d,aq-pb)$ divides $a\ell$. Analogously, $e$ divides $b\ell$ too and hence $e|\ell$ because
$\gcd(d,a,b)=1$.
\end{proof}

\subsection{Curvettes, valuations, and generic germs}
\label{sec-curvettes}
In this section we will fix the surface singularity $X=X(d;1,q)$ and the Hirzebruch-Jung resolution $\pi$ described above,
which is a composition of $n$ weighted blow-ups centered at singular points.
Consider $E_i$ the exceptional component obtained at the $i$-th blow-up according to~$\pi$.
Following Deligne~\cite{Deligne-intersections}, an \emph{$E_i$-curvette} on $X$ is the image of a smooth curve 
transversal to $E_i$ at a smooth point.

The Hirzebruch-Jung resolution introduced in Section~\ref{sec-quotient} defines valuations $v_i:\cO_X^* \to \ZZ$ associated 
with each exceptional divisor $E_i$, $i=1,...,n$ by calculating the intersection multiplicity of a germ $f\in \cO_X^*$
with $E_i$ in the resolution process. Note that this definition can be naturally to $\cO_X^*(k)$ as follows: 
$v_i(f):=\frac{1}{d}v_i(f^d)$, where $f\in \cO_X^*(k)$ and thus $f^d\in \cO_X^*$. This results into a family of morphisms:
$v_i:\cO_X^*(k) \to \frac{1}{d}\ZZ$ satisfying 
$$
v_i(h\cdot f)=v_i(h)+ v_i(f), \quad \quad \forall\ h\in \cO_X^*, f\in \cO_X^*(k),
$$
and
$$
v_i(f+g)\geq \min\{v_i(f),v_i(g)\}, \quad \quad \forall\ f,g\in \cO_X^*(k).
$$
We will denote by $v=\sum v_i:\cO^*_X(k)\to (\frac{1}{d}\ZZ)^n$ the morphism $v(f)=(v_i(f))$.

\begin{dfn}
A $\Q$-divisor $D=\{f=0\}$, $f\in \cO_X^*(k)$ is called \emph{generic} if $v(f)$ is a minimal element in 
$v(\cO_X^*(k))\subset (\frac{1}{d}\ZZ)^n$ with its induced partial order, that is, $f$ is minimal if $v(g)\leq v(f)$ implies $v(g)=v(f)$ for any $g \in \cO_X^*(k)$.
\end{dfn}

\subsection{Newton polygon, Milnor number, and \texorpdfstring{$\delta$}{delta}-invariant}
\label{sec-newton}
Let $D=\{f=0\}$ be a $\Q$-divisor with $f\in \cO_X(k)$. Define the \emph{Newton diagram} of $f=\sum a_{\ii,\j}x^\ii y^\j$ as 
$$N_\LL(f)=\{(\ii,\j)\in \LL\mid a_{\ii,\j}\neq 0\}\subset \LL(k),$$
where $\LL(k):=\{(\ii,\j)\in \NN^2\mid \ii+q\j\equiv k \mod d\}$ and $\LL:=\LL(0)$ is the structure lattice.
The convex hull of $N_\LL(f)+\LL$ is called the $\LL$-\emph{Newton polygon} of $f$ and denoted by~$\Gamma_\LL(f)$.
This extends the notion of Newton polygon given in~\cite{Reid-Surface} for functions on~$X$, that is, for what we
refer here as the structure lattice. 
The following properties are an immediate consequence of the definitions:

\begin{prop}\label{proper-newton} 
\mbox{}
\begin{enumerate}
 \item\label{proper-newton-1}
$N_\LL(f_1f_2)=N_\LL(f_1)\oplus N_\LL(f_2)$, where $\oplus$ denotes the Minkowski sum,
 \item\label{proper-newton-2}  
The number of compact faces of $\Gamma_\LL(f)$ is an upper bound of the number of irreducible branches of~$f$,
 \item\label{proper-newton-3}
$\bigcup_{f\in \cO_X(k)} N_\LL(f)=\LL(k)$,
 \item\label{proper-newton-4} 
$N_\LL(f)\subset N_\LL(g) \Rightarrow v(g)\leq v(f)$.
\end{enumerate}
\end{prop}

As a consequence one obtains the following interpretation of generic $\Q$-divisors.

\begin{prop}\label{prop-Nk} 
If $f\in \cO_X(k)$ is a generic $\Q$-divisor, then $\Gamma_\LL(f)=\LL(k)$.
\end{prop}

In~\cite{CMO12} we extended the concept of Milnor fiber and Milnor number of a curve singularity 
allowing the ambient space to be a quotient surface singularity. 
Alternative generalizations of Milnor numbers can be found, for instance, 
in~\cite{brasselet-milnor,STV-Milnornumbers,Le-Someremarks}.
The Milnor number proposed here seems natural for surfaces and allows for a generalization of the local 
$\delta$-invariant and can be described in terms of a $\QQ$-resolution of the curve singularity.

\begin{dfn}[\cite{CMO12}]
Let $D=\{f=0\} \subset X(d;1,q)$ be a reduced Weil divisor. 
The \emph{Milnor fiber} $(f_t)_X$ of $D$ is defined as follows,
$$
(f_t)_X:=\{ f=t\}/{\qa_d}.
$$
The \emph{Milnor number} $\mu_X$ of $D$ is defined as
$$
\mu_X:=1-\chi^{\orb}(f_t)_X.
$$
Define the \emph{delta invariant} $\delta_X$ of $f$ as the rational number verifying 
\begin{equation*}
\chi^{\orb}(f_t)_X=r_X(f) - 2 \delta_X(f),
\end{equation*}
where $r_X(f)$ is the number of local branches of $D$ at $0$
and $\chi^{\orb}(f_t)_X$ denotes the orbifold Euler characteristic of~$(f_t)_X$.
\end{dfn}

\begin{rem}
\label{rem-delta1}
Note that, with this definition, $\delta_X(u)=0$ for $u\in \cO^*_X$ a unit in the ring of functions on~$X$.
\end{rem}

The following formula for the $\delta$-invariant of the product will be useful in the future.

\begin{lemma}[{\cite[Corollary 4.8]{CMO12}}]
\label{lem-deltafg}
For any $f,g$ reduced quasi-invariant $\Q$-divisor on~$X$, the following holds
$$\delta_X(fg) = \delta_X(f) + \delta_X(g) + (f,g)_X,$$
where $(f,g)_X$ denotes the intersection multiplicity of $f$ and $g$ in~$X$.
\end{lemma}

Consider $\overline{\LL}$ an integral lattice and $\overline{\LL}_p=p+\overline{\LL}\subset \ZZ^2$ an affine integral lattice. 
Denote by $d$ the absolute value of the determinant of a basis of $\overline{\LL}$, which is an invariant of $\overline{\LL}_p$. 
Let us denote by $\LL_p$ (resp.~$\LL$) the restriction of $\overline{\LL}_p$ (resp.~$\overline{\LL}$) to the first quadrant
$(\ZZ_{\geq 0})^2$. 
We say a polygon $N$ in $\LL_p$ is an \emph{$\LL$-Newton polygon} if $a+\LL\subset N$ for any $a\in N$. 
Moreover, we will say $N$ is \emph{convenient} if both $(\ZZ_{\geq 0}\times \{0\})\setminus N$ and 
$(\{0\}\times \ZZ_{\geq 0})\setminus N$ are finite. Note that $\LL_p$ itself is a convenient $\LL$-Newton polygon.
Given a convenient $\LL$-Newton polygon $N$ in $\LL$ one can consider $\Gamma(N)$ the convex hull of $N$.
Following Kouchnirenko's definition one can consider the \emph{$\LL_p$-Newton number} of a convenient $\LL$-Newton polygon $N$ 
in $\LL_p$ as
\begin{equation}
\label{eq-muN}
\mu_{\LL_p}(N)=2V_N-V_{1,N}-V_{2,N}+\mu_{\LL_p},
\end{equation}
where $V_N$ is the area of $\Gamma(\LL_p)\setminus \Gamma(N)$ divided by~$d$. Analogously, $V_{i,N}$ is the 
length of $\Gamma(\LL_p)\setminus \Gamma(N)$ on the hyperplane $x_i=0$ divided by $d$ (see Figure~\ref{fig-mu}). 

\begin{figure}[ht]
\begin{center}
\begin{tikzpicture}[scale=.5]
    \draw [<->,thick] (0,8) node (yaxis) [above] {$ $}
        |- (11,0) node (xaxis) [right] {$ $};
    \draw [thick] (10,0) -- (7,.5);
    \fill (7,.5) circle (2pt);
    \draw [dashed] (7,.5) -- (5,1);
    \fill (5,1) circle (2pt);
    \draw [thick] (5,1) -- (4.5,1.3);
    \draw [dashed] (4.5,1.3) -- (4,1.6);
    \draw [thick] (4,1.6) -- (3.5,1.9);
    \fill (3.5,1.9) circle (2pt) node [above right] {$N$};
    \draw [thick] (3.5,1.9) -- (2.5,2.9);
    \draw [dashed] (2.5,2.9) -- (1.5,3.9);
    \draw [thick] (1.5,3.9) -- (.5,5);
    \fill (.5,5) circle (2pt);
    \draw [dashed] (.5,5) -- (0,7);
    \draw [thick] (5,0) -- (2,.5);
    \fill (2,.5) circle (2pt) node [left] {$\mathbb{L}_p$};
    \draw [dashed] (2,.5) -- (.5,2);
    \fill (.5,2) circle (2pt) node [right] {$\quad V_N$};
    \draw [thick] (.5,2) -- (0,3.5);
\end{tikzpicture}
\caption{}
\label{fig-mu}
\end{center}
\end{figure}

The term $\mu_{\LL_p}$ needs a more 
careful explanation. Note that the convex hull $\Gamma(\LL_p)$ of $\LL_p$ does not necessarily contain the origin $\{x_1=x_2=0\}$. 
However, note that $\LL_p$ it must be a convenient $\LL_p$-Newton polygon. The invariant $\mu_{\LL_p}$ is defined as
\begin{equation}
\label{eq-muL}
\mu_{\LL_p}:=
\begin{cases}
-1 & \text{ if } \LL_p=\LL\\
\frac{d-1}{d}+\frac{\mu(\Gamma(\LL_p))}{d}=
1+\frac{2V-V_1-V_2}{d}&\text{ otherwise, }
\end{cases}
\end{equation}
where $\mu(\Gamma(\LL_p))$ denotes the standard Newton number of $\Gamma(\LL_p)$ and $V$, $V_1$, $V_2$ denote the standard 
volumes of the compact region under~$\Gamma(\LL_p)$.

In case $N=N_\LL(f)$ is the $\LL$-Newton diagram of a non-degenerate germ $f\in \cO_X(k)$, note that $\LL(k)$ is also
a convenient $\LL$-Newton polygon and one obtains immediately the 
following generalization of Kouchnirenko's Theorem~\cite{Kouchnirenko-Newton}.

\begin{thm}
\label{thm-mu}
If $N=N_\LL(f)$ is the $\LL$-Newton polygon of a non-degenerate germ $f\in \cO_X(k)$, then
$$
\mu_X(f)=\mu_{\LL(k)}(N).
$$
\end{thm}

Also, using Pick's Theorem in~\eqref{eq-muN} one obtains $2V_N=B_N+2I_N-2$, where $B_N=V_{1,N}+V_{2,N}+\|k\|_1+r_N$ is 
the number of $\LL$-points on the boundary of $\Gamma(\LL_p)\setminus \Gamma(N)$,
$r_N$ is the number of compact segments in $N$, and $I_N$ is the number of $\LL$-points in the interior of 
$\Gamma(\LL_p)\setminus \Gamma(N)$, see Figure~\ref{fig-mu}. Therefore Theorem~\ref{thm-mu} can be written~as
\begin{equation}
\label{eq-mu2}
\mu_X(f)-\mu_{\LL(k)}=2I_N+\|k\|_1+r_N-2.
\end{equation}

The invariants $\K_{X}$ described in~\eqref{eq-k0} and $\delta_X$ can be combined to define an invariant of the 
Weil divisor class as follows.

\begin{propdef}[\cite{Ortigas13PhD,CMO14}]
\label{propdef-Delta}
Let $X=X(d;1,q)$ and $f\in \cO_X(k)$. Then $\R_X(f):=\delta_X(f)-\K_{X}(f)$ defines a map
$$\R_X:\Cl(X)\cong \qa_d\to \QQ,$$
that is, $\R_X(f)=\R_X(g)$ for any $f,g\in \cO_X(k)$.

Moreover, $\R_X(0) = 0$ and $\R_X(f) = \frac{d-1}{2d}$ if $f$ defines a quasi-smooth $\Q$-divisor.
\end{propdef}

For the definition of quasi-smooth divisors we refer to~\cite{Dolgachev82}.

Assuming $D=\{f=0\}$, where $f\in \cO_X(k)$, the isomorphism $\Cl(X)\cong \ZZ_d$ will allow us to 
write $\R_X(f)$, $\R_X(D)$ or simply $\R_X(k)$ depending on the context.

The map $\R_X$ is equivalent to the map $R_X$ described in~\cite[p.312]{Blache-RiemannRoch}.
More precisely, $R_X(k) = -\R_X(-k)$ holds. 
We expect that our approach will allow us to solve some of the conjectures stated there.

Also note that $\R_X:\Cl(X) \to \QQ$ characterizes a normal surface singularity in the following sense.

\begin{prop}
Let $D \in \Cl(X)$ be a quasi-smooth Weil divisor in a surface singularity $X$ and 
denote by $d_k:=\R_X(kD)$. Then, $d_0=0$ and
$$
X\cong X(d;1,q),
$$
where $d:=\frac{1}{1-2d_1}$ and $q:=dd_2+1$.
In particular, $\R_X(D)$ and $\R_X(2D)$ characterize the quotient singularity~$X$.
\end{prop}

\begin{proof}
Since $\R_X(0)$ and $\R_X(D)$ have already been discussed in Proposition-De\-fi\-ni\-tion~\ref{propdef-Delta},
let us calculate $\R_X(2D)$ in $X=X(d;1,q)$. Since $X(d;1,q)$ and $X(d;1,\bar q)$ are both the same quotient space,
replacing $q$ by $\bar q$ if necessary, one can assume that $D \in \cO_X(1)$. According to the definition
$\R_X(2D) = \delta_X(f) - \K_X(f)$ where $f$ is any element in $\mathcal{O}_X(2)$.
For instance consider $f=(x+y^{\bar{q}})(x-y^{\bar{q}})$. Then Lemma~\ref{lem-deltafg} implies
$$
  \delta_X(f) = 
  \frac{d-1}{2d} + \frac{d-1}{2d} + \frac{\bar{q}}{d},
$$
since the $\delta$-invariant of a quasi-smooth divisor is $\frac{d-1}{2d}$ (\cite[Remark 4.7]{CMO12}) 
and their multiplicity of intersection is $\frac{\bar{q}}{d}$. In order to compute $\K_X(f)$, let us blow-up the origin of 
$X(d;1,q)$ with weights $(1,\bar{q})$. In this case the blow-up is a $\mathbb{Q}$-resolution and 
hence Proposition~\ref{prop-KP-blowup} tells us that
$$
\K_X(f) = \# \{ (i,j) \mid i,j \geq 1, \ 2 \bar{q} \geq \bar{q} i + j \equiv 2 \bar{q} \!\! \mod d \}.
$$
There is only one point in this set, namely $(1,\bar{q})$. Therefore $\R_X(2D) = \frac{\bar{q}-1}{d}$.
\end{proof}

\section{Arithmetics for generic \texorpdfstring{$\Q$}{Q}-divisor}
\label{sec-arithmetics}
In this section we will define the basic arithmetic data associated with the cyclic quotient singularity $X$
and we will describe a generic germ for a given divisor class~$k$. This will be useful to describe the 
invariant~$\R_X(k)$.


\subsection{Further numerical properties of quotient surface singularities}
\label{sec-num-proper}
Consider $X=X(d;1,q)$ a cyclic quotient surface singularity and 
$\mathbf{q}=[q_0=d,q_1=q,q_2,...,q_{n-1},q_n=1,q_{n+1}=0]$,
$\mathbf{c}=[c_0=2,c_1,c_2,...,c_{n-1},c_n=q_{n-1},c_{n+1}=2]$
as described in Section~\ref{subsec-wblowup} such that $q_{i-1}=c_iq_{i}-q_{i+1}$ for $i=1,\dots,n$ (the remaining
coefficients $q_{n+1},c_0,c_{n+1}$ are defined for convenience). Note that $c_i\geq 2$ for all $0\leq i\leq n+1$.
Define $\bar q_i$ as the smallest positive integer such that $q_1\bar q_i\equiv q_i \mod d$.
This way one defines 
$\mathbf{\bar q}=[\bar q_0=0,\bar q_1=1,\bar q_2,...,\bar q_{n-1},\bar q_n,\bar q_{n+1}=d]$.

We will describe some useful properties relating $\mathbf{q}$, $\mathbf{\bar q}$, and $\mathbf{c}$ which 
will be used in the upcoming sections.

\begin{lemma}
\label{lemma-q}
Let $X_i = X(q_i;1,q_{i+1})$ and $(\bar{q}_{X_i})_j$ denote the corresponding $\bar{q}_j$ associated with
the space $X_i$. Under the conditions above one has the following:
\begin{enumerate}
\item \label{lemma-q-1} $\bar{q}_i = d q_i \left( \frac{1}{q_0 q_1} + \cdots + \frac{1}{q_{i-1} q_{i}} \right)$,
\item \label{lemma-q-2} $\bar{q}_i = c_{i-1} \bar{q}_{i-1} - \bar{q}_{i-2}$,
\item \label{lemma-q-3} $\bar{q}_j q_i - q_j \bar{q}_i = d \cdot (\bar{q}_{X_i})_{j-i}$, $\forall j \geq i$,
\item \label{lemma-q-4} $q_i = q \bar{q}_i - d \cdot (\bar{q}_{X_1})_{i-1}$,
\item \label{lemma-q-5} $\bar{q}_{i+1} q_i - q_{i+1} \bar{q}_i = d$,
\item \label{lemma-q-6} $(\bar{q}_{X_i})_{j} (\bar{q}_{X_{i+1}})_{j} - (\bar{q}_{X_i})_{j+1}(\bar{q}_{X_{i+1}})_{j-1}= 1$.
\end{enumerate}
\end{lemma}

\begin{proof}
Let $\overline{Q}_i := d q_i \left( \frac{1}{q_0 q_1} + \cdots + \frac{1}{q_{i-1} q_{i}} \right)$.
Formulas~\eqref{lemma-q-2}--\eqref{lemma-q-6},
replacing $\bar{q}_i$ by $\overline{Q}_i$, are easily checked after some simple calculations.
Hence $q_1 \overline{Q}_i \equiv q_i \mod d$.
To end the proof it is enough to show $\overline{Q}_i < d$ for $i=1,\ldots,n$. Let us fix $1 \leq i \leq n$,
since $q_{j} > q_{j+1}$, one obtains $q_j \geq q_i + (i-j)$ for $j=0,1,\ldots,i$. Therefore,
\begin{align*}
\overline{Q}_i & \leq d q_i \left( \frac{1}{(q_i+i)(q_i+i-1)} + \frac{1}{(q_i+i-1)(q_i+i-2)} +
\cdots + \frac{1}{(q_i+1) q_i} \right) \\
& = d q_i \left( \frac{2}{(q_i+i)(q_i+i-2)} + \cdots + \frac{1}{(q_i+1) q_i} \right) = \cdots =
d \cancel{q_i} \frac{i}{(q_i + i) \cancel{q_i}} < d. 
\end{align*}
\end{proof}

\subsection{The \texorpdfstring{$X$}{X}-decomposition of \texorpdfstring{$k$}{k} and the coin change-making problem}
\label{sec-Xdecomp}
Consider now $\alpha\in (\ZZ_{\geq 0})^{n+2}$ a vector
with $n+2$ non-negative coordinates $\alpha=[\alpha_0,\dots,\alpha_{n+1}]$, we will define
$\|\alpha\|_X:=\alpha\cdot \mathbf{q}$. In this context, if $k\equiv \|\alpha\|_X \mod d$, we will say 
$\alpha$ is an $X$-\emph{decomposition} of $k$. We will also define $\|\alpha\|_1:=\sum \alpha_i$.

\begin{rem}
\label{rem-Xdecomp}
Consider $D=\{f=0\}$, with $f\in\cO_X(k)$, a $\QQ$-divisor in $X$ and 
$v:\cO^*_X(k)\to \frac{1}{d}(\ZZ^n)$ the list of valuations as defined in Section~\ref{sec-curvettes}.
The vector $\alpha(f):=[0,d v_1(f),\dots,d v_n(f),0]$ has integral coordinates. 
Note that $\|\alpha(f)\|_X\equiv k \mod d$.
\end{rem}

From yet another point of view, an $X$-\emph{decomposition} of $k$ is a solution to the following 
change-making scenario: given an integer $k$ and a sequence of coin values $\mathbf q=[q_0,q_1,\dots,q_{n}]$ 
one is interested in the amount of coins of each type $[k_0,k_1,\dots,k_{n}]$ that add up to $k$, 
that is, $\sum_i k_iq_i=k$. 

Among all possible solutions to the coin change-making scenario, there is an effective one following 
the \emph{greedy algorithm} resulting from picking the largest value coin which is not greater than the 
remaining amount. In our case, this results in the following.

\begin{dfn}
\label{def-greedy}
Let $X=X(d;1,q)$ be a surface and $\mathbf q$ defined as above, then the \emph{\greedy} $X$-decomposition of
$k$ is the following list of integers $[k_0,k_1,...,k_n]$, resulting from the quotients of the 
division of $0\leq k'<d$, $k\equiv k' \mod d$, by $\mathbf{q}$, that is, $k'=k_1q_{1}+k'_1$, and 
$k'_{i}=k_{i+1}q_{i+1}+k'_{i+1}$ for $i\geq 1$.
\end{dfn}

The \greedy $X$-decomposition of $k$ will be denoted by~$[k]$.
Note $\|[k]\|_1:=\sum k_i=[k]\cdot \mathbf{1}$ which will be simply denoted by~$\|k\|_1$.

Also, among all possible solutions to the coin change-making scenario, one can state the following
\emph{knapsack type problem} called the \emph{coin change-making problem}. 

\begin{problem}[Coin Change-Making Problem]
\label{prob-coin}
Given $k$ and $\mathbf q$, find a solution $\alpha$ to the coin change-making scenario which minimizes 
the number of coins, that is, such that $\|\alpha\|_1$ is minimal.
\end{problem}

The greedy algorithm does not provide in general a solution to the coin change-making problem, for instance,
for $k=6$ and $\mathbf q=[4,3,1]$ note that the greedy algorithm provides the following solution to the change-making 
scenario $[6]=[1,0,2]$ which is not a solution to the problem since $[0,2,0]$ uses fewer coins.

In our case however the answer is positive.

%

\begin{lemma}
\label{lemma-greedy}
Given a surface $X=X(d;1,q)$ and $\mathbf q$ as above, the \greedy \break $X$-decomposition of $k\in \ZZ$ is a
solution to the Coin Change-Making Problem~\ref{prob-coin}.
\end{lemma}

\begin{proof}
The result is a direct consequence of the proof of the main result in~\cite{MR0472020} (see also~\cite[Theorem p.4]{andy}) which we 
summarize here for convenience. Denote by $\Opt(\mathbf q,k)$ (resp.~$\G(\mathbf q,k)$) the number of coins in a solution 
(resp.~\greedy candidate) to the coin change-making problem for $\mathbf q$ and $k$. Denote by 
$c_i:=\lceil \frac{q_{i-1}}{q_i}\rceil$.
Then $\Opt(\mathbf q,k)=\G(\mathbf q,k)$ for all $k$ if $\G(\mathbf q,c_iq_i-q_{i-1})\leq c_i-1$ for all $i=1,\dots,n$.
In our situation, $c_iq_i-q_{i-1}=q_{i+1}$ and $c_i\geq 2$, hence $\G(\mathbf q,c_iq_i-q_{i-1})=1$ and the result follows.
\end{proof}

\subsection{Proof of Theorem~\ref{thm-main-generic}}
Before we start with the proof, let us describe the irreducible curvettes in~$X$.

\begin{lemma}
\label{lemma-qi}
The quasi-invariant $\Q$-divisor $f=x^{q_i}-y^{\bar q_i}$ is a curvette in~$\cO_X(q_i)$.
\end{lemma}
\begin{proof}
In order to show this result, we will use a recursive argument on the length of the canonical resolution
of $X$ (see Section~\ref{subsec-wblowup}). Let us perform the $(1,q)$-blow-up of $X$. The strict
transform of $f$ is $\widehat{f} = x^{q_i} + y^{\frac{q \bar{q}_i-q_i}{d}} \in \O_{\widehat{X}}(q_i)$,
where $\widehat{X} = X(q;1,q_2)$. By Lemma~\ref{lemma-q}\eqref{lemma-q-4},
$\widehat{f} = x^{q'_{j}}-y^{\bar{q}'_{j}}$ where $q' = q_{\widehat{X}}$ and $j=i-1$.

Hence it is enough to check the result for $q_1$, equivalently that
the Newton polygon $\Gamma_\LL(q_1)$ has only one compact face. 
This is immediate since $\Gamma_\LL(q)$ is the $\LL(k)$-convex hull of 
$\{(0,1)\}+\LL$ and $\{(q,0)\}+\LL$ as in Figure~\ref{fig-Oq} which has only one compact face.
\begin{figure}[ht]
\begin{center}
\begin{tikzpicture}[scale=1]
    \draw [<->,thick] (0,2) node (yaxis) [above] {$ $}
        |- (4,0) node (xaxis) [right] {$ $};
    \fill (0,1) circle (2pt) node [left] {$(0,1)$};
    \draw (3,0) -- (0,1);
    \fill (3,0) circle (2pt) node [below] {$(q,0)$};
    \draw (2.5,1.5) node [left] {$N_\LL(q)$};
\end{tikzpicture}
\end{center}
\caption{}
\label{fig-Oq}
\end{figure}
\end{proof}

\begin{rem}
Note that $f$ as above might not be irreducible as a germ in $\cO_{\CC^2}$. For instance,
in $X=X(6;1,5)$ one has $\mathbf q=[6,5,4,3,2,1,0]$, $\mathbf{\bar q}=[0,1,2,3,4,5,6]$ and thus 
by Lemma~\ref{lemma-qi}, $x^4-y^2\in \cO_X(4)$ is irreducible in~$\cO_X(4)$. Note that neither
$x^2-y$ nor $x^2+y$ are quasi-invariant.
\end{rem}

\begin{proof}[Proof of Theorem~\ref{thm-main-generic}]
Following the notation introduced in Remark~\ref{rem-Xdecomp} note that $\|\alpha(f)\|_X=k$.
Assume $g\in \cO_X(k)$ is a germ such that $v(g)\leq v(f)$. To show the minimality of $f$ it is enough 
to prove that $v(g)=v(f)$. On the one hand $k\equiv \|\alpha(g)\|_X\leq \|\alpha(f)\|_X=k$ implies
$\|\alpha(g)\|_X=k$. On the other hand, by Lemma~\ref{lemma-greedy} $\alpha(f)$ is a solution to the 
coin change-making problem associated with $\|\alpha\|_X=k$, that is, $\|\alpha(f)\|_1$ is minimal.
Since $\|\alpha(g)\|_1\leq \|\alpha(f)\|_1$, one has $\|\alpha(g)\|_1= \|\alpha(f)\|_1$, which together
with $v(g)\leq v(f)$, implies $v(g)=v(f)$.

The \emph{moreover} part is equivalent to proving $N_\LL(g)=N_\LL(f)$ for any generic germ $g\in \cO_X(k)$.
This is a consequence of Proposition~\ref{prop-Nk} since $N_\LL(f)=\LL(k)=\LL(g)$.
\end{proof}

\subsection{Proof of Theorem~\ref{thm-Delta}}

\begin{proof}
Since $\R_X(k)=\delta_X(f)-\K_X(f)$ does not depend on the choice of $f\in \cO_X(k)$ 
(see~\cite{CMO14,Ortigas13PhD}), one can calculate this invariant using the generic germ provided in 
Theorem~\ref{thm-main-generic}. On the other hand, the recursive formula for $\delta_X(f)$ is 
given in~\cite{CMO12}. Therefore, it is enough to compute $\K_X(f)$ for the generic germ $f$.
This is a consequence of the following result.
\end{proof}

\begin{lemma}
\label{lemma-K-generic}
If $f\in \cO_X(k)$, $k\neq 0$ is a generic germ, then $\K_{X}(f) = \| k \|_1 -1 $, 
that is, the number of its irreducible components minus one.
\end{lemma}

\begin{proof}
By the second part of Theorem~\ref{thm-main-generic} we can assume that $f$ has the form given 
in~\eqref{eq-fgeneric}. Consider the $(1,q)$-weighted blow-up at
$0 \in X$ and denote by $\widehat{D}$ the strict transform of $D:=\{f=0\} \subset X$.
By Proposition~\ref{prop-KP-blowup},
\begin{equation}\label{formula_KPf2}
\K_{X}(D) = \K_{\pi} + \sum_{P \in E \cap \widehat{D}} \K_P (\widehat{D}),
\end{equation}
where $\K_{\pi}:= \# \{ (i,j) \in \mathbb{Z}^2 \mid i,j \geq 1, \, k \geq i +qj \equiv k \! \mod d \}$.
The unique pairs $(i,j)$ satisfying both the inequality and the congruence are
$(k-q,1), (k-2q,2), \ldots, (k-jq,j)$ as long as $k-jq > 0$. Then,
$$
\K_\pi =
\begin{cases}
k_1 - 1 & \text{if} \ k=k_1 q, \\
k_1 & \text{otherwise}.
\end{cases}
$$
On the other hand, by construction, $\K_P(\widehat{D}) \neq 0$ only when $P$
is the singular point of type $(q;1,q_2)$.

Equation~\eqref{formula_KPf2} allows us to repeat the same arguments until $X$ is smooth,
see Section~\ref{subsec-wblowup}. The conclusion is that $\K_{X}(f) = \sum_{i=1}^n k_i - 1 = \|k\|_1-1$,
that is, the number of local branches of $f$ minus~$1$.
\end{proof}

\subsection{Proof of Theorem~\ref{thm-descomp}}

\begin{proof}
Since $\cO_X(q_0)^{k_0}=\cO_X$, without loss of generality, we can assume~$0\leq k <d$.
\begin{enumerate}
\item[\textbf{Case 1.}]
Note that if $k=q_i$ there is nothing to prove.
\item[\textbf{Case 2.}]
If $[k]=[0,...,0,k_i,0,...,0]$ it is enough to show that $I:=\dim \frac{\cO_X(k)}{\otimes \cO_X(q_i)^{k_i}}=0$,
that is, there are no $\LL(k)$-points under the polygon $\oplus_i \Gamma_\LL(q_i)^{k_i}$. In order to do this consider 
$f=(x^{q_i}-\lambda_1 y^{\bar q_i})\cdots (x^{q_i}-\lambda_{k_i} y^{\bar q_i})$. 
 
The $(\bar q_i,q_i)$-blow-up $\pi$ of $X$ is a $\Q$-resolution of $f$ and thus, by Proposition~\ref{prop-KP-blowup},
\begin{equation}
\label{eq-Kpi}
\K_{X}(f) = \K_{\pi} + \sum_{P \in E \cap V(\widehat{f})} \K_P (\widehat{f})=\K_{\pi},
\end{equation}
where $\K_{\pi} = \# \{ (\ii,\j) \in \mathbb{Z}^2 \mid \ii,\j \geq 1, \ k \bar{q}_i \geq \bar{q}_i r + q_i s \equiv k \bar{q}_i
\! \mod d \}$.

Note that both $I$ and $\K_\pi$ describe a certain number of $\LL(k)$-points as follows: let $s$ denote the number of 
points on the compact segment $L$ of $\oplus_{k_i} \Gamma_\LL(q_i)$, $I_1$ (resp.~$I_2$) the number of points on the 
$x$-axis (resp.~$y$-axis) under or on $L$. Then $I$ and $\K_\pi$ are related by
\begin{equation}
 \label{eq-cardinals}
I+s=\K_\pi+I_1+I_2.
 \end{equation}
On the other hand it is clear that $s=k_i+1$. 
Also, $I_1=1$ (and analogously by symmetry $I_2=1$). Therefore
formula~\eqref{eq-cardinals} becomes $\K_{\pi}=k_i-1+I$.

Finally, by Lemma~\ref{lemma-K-generic} one has $\K_{X}(f)=k_i-1$. Since $\K_{X}(f) = \K_\pi$ by~\eqref{eq-Kpi},
one obtains~$I=0$.
 \item[\textbf{Case 3.}]
 In general, since $0\leq k<d$ one has the \greedy $X$-decomposition of $k$, namely $[k]=[0,k_1,...,k_n]$. Again, we will show that 
 $I:=\dim \frac{\cO_X(k)}{\otimes \cO_X(q_i)^{k_i}}=0$. Assume $k_1\neq 0$, otherwise the result will be proved by
 induction. Consider $x^{\ii_0}y^{\j_0}$ a monomial in $\cO_X(k)$, that is, $0\leq \ii_0+q\j_0\equiv k \mod d$. 
 Note that the slope of the compact segment in $\Gamma_\LL(q_i)$ is $-\frac{\bar q_i}{q_i}$, hence the biggest slope among the
 compact segments in 
 ${\displaystyle{\bigoplus_{i,k_i} \Gamma_\LL(q_i)}}$ is $-\frac{\bar q_1}{q_1}$, and its corresponding line has equation $i+qj=k$.
 Note that $x^{\ii_0}y^{\j_0}$ must be such that $k\leq \ii_0+q\j_0\equiv k \mod d$. Therefore, after substituting
 $x\mapsto x$, $y\mapsto x^qy$ one can construct the following monomial $x^k (x^{\ii_1}y^{\j_1})$,
 where $\ii_1=(\ii_0+q_1\j_0-k)/d$, $\j_1=\j_0$, $\ii_1+q_2\j_1\equiv k^{(1)} \mod q$, $k=k_1q_1+k^{(1)}$, and 
 $x^{\ii_1}y^{\j_1}\in \cO_{X_1}(k^{(1)})$, $X_1=X(q_1;1,q_2)$. Note that, given $x^{\ii_1}y^{\j_1}\in \cO_{X_1}(k^{(1)})$, 
 one can recover the original monomial in $\cO_X(k)$ by writing $\j_0:=\j_1$ and $\ii_0:=\ii_1d+k-q\j_1$. Recursively,
 at the last step, one can use Case 2 to prove that such a monomial must belong to $\cO_{X_{n'}}(k^{(n')})$, where
 $n'$ is the last non-zero entry of $[k]$, namely, $[k_0,k_1,...,k_{n'},0,...,0]$. Hence, 
 this implies that $x^{\ii_0}y^{\j_0}\in \otimes \cO_X(q_i)^{k_i}$.
\end{enumerate}
\end{proof}

%
%

\subsection{Calculation of Newton numbers}
To end this section we will give some formulas to calculate the Newton numbers $\mu_X(k):=\mu_{\LL(k)}$ of a quotient 
surface singularity $X=X(d;1,q)$ as defined in~\eqref{eq-muL}.
In order to do so we need to introduce some notation. Recall the invariants $\mathbf{q}$, $\mathbf{\bar q}$ introduced 
in Section~\ref{sec-num-proper} as well as $[k]=[k_0,k_1,\dots,k_n,k_{n+1}]$ the greedy $X$-decomposition of $k$ as defined 
in Section~\ref{sec-Xdecomp}. If $0\leq k<d$, then $k=[k]\cdot \mathbf{q}$. We define $\bar k:=[k]\cdot \mathbf{\bar q}$. Also, 
denote by $Q=(Q_{ij})$ the following $(n+1)$-square upper triangular matrix: 
$$Q_{ij}:=
\begin{cases} 
0 & \text{ if }\ 0\leq j< i\leq n\\
(\bar q_{X_i})_{j-i+1} & \text{ if }\ 0\leq i\leq j\leq n,
\end{cases}
$$
where $X_i:=X(q_i;1,q_{i+1})$. Let us also use the notation $[0,k]$ (resp.~$[k,0]$) referring to $[0,k_1,\dots,k_n]$
(resp.~$[k_1,\dots,k_n,0]$). One has the following result.

\begin{prop}
\label{prop-muk}
$$\mu_X(k)=
\begin{cases}
 -1 & \text{ if }\ k=0\\
 \frac{(d-1)+(k-1)(\bar k -1)}{d}-[0,k]Q[k,0]^t  & \text{ otherwise.}
\end{cases}
$$
In particular, $\mu_X(q_i)=\mu_X(\bar q_i)=\frac{(d-1)+(q_i-1)(\bar q_i-1)}{d}$.
\end{prop}

\begin{proof}
According to~\eqref{eq-muL}, we only need to prove the case $0<k<d$, in which $\LL(k)\neq \LL$. Note that
$$
\mu_X(k)=1+\frac{2V-V_1-V_2}{d},
$$
where $2V=\frac{\sum q_i\bar q_i k_i^2+2\sum_{i>j} q_i\bar q_j k_ik_j}{d}$, $V_1=k$, and $V_2=\bar k$.
Finally, by Lemma~\ref{lemma-q}\eqref{lemma-q-3} one obtains
$2V=k\bar k-k-\bar k-d\sum_{i>j} (\bar q_{X_j})_{i-j} k_ik_j$. Therefore
$$
\mu_X(k)=1+\frac{(k-1)(\bar k-1)-1}{d}-\sum_{i>j} (\bar q_{X_j})_{i-j} k_ik_j,
$$
which results in the required formula.
In the particular case when $k=q_i$, then $k_i=1$ and $k_j=0$ if $j\neq i$, that is, $[k]=[0,\ldots,\overset{i}{1},\ldots,0]$. 
Therefore $[0,k]Q[k,0]^t=0$ and hence $\mu_X(q_i)=\frac{(d-1)+(q_i-1)(\bar q_i-1)}{d}$. 
By symmetry, since $X(d;1,q)\cong X(d;1,\bar q)$
one also obtains $\mu_X(\bar q_i)=\frac{(d-1)+(q_i-1)(\bar q_i-1)}{d}$.
\end{proof}

In light of~\eqref{eq-mu2} one has the following.
\begin{cor}
If $f\in \cO_X(k)$, $k\neq 0$ is a non-degenerated germ in $X=X(d;1,q)$, then
$$
\mu_X(f)=2I_N+\|k\|_1+r_N+\frac{(k-1)(\bar k -1)}{d}-\frac{d+1}{d}-[0,k]Q[k,0]^t,
$$
depends only on its $\LL(k)$-Newton polygon.
\end{cor}

\begin{rem}
The matrix $Q$ can easily be constructed from $\mathbf{q}$ and $\mathbf{c}$. First note that $Q$ is an upper 
triangular $(n+1)$-square matrix with 1's on the diagonal. The last column is given by 
$[q_0=d,q_1=q,q_2,\dots,q_n]$ and the elements over the diagonal are given by $[c_1,\dots,c_n]$. The rest of 
the matrix can be filled using Lemma~\ref{lemma-q}\eqref{lemma-q-6}, which translates into the property that 
every two-by-two minor of $Q$ on the (upper triangular part of $Q$) has determinant equal to~1. 
Starting from bottom to top one can easily reconstruct the matrix $Q$.
\end{rem}

\begin{exam}
\label{ex-X14-11-mu}
Consider $X=X(14;1,11)$, then
$$
\array{rclcccccc}
\mathbf{q} &=& {[}14,&11,&8,&5,&2,&1,&0{]}\\
\mathbf c &=& {[}2,&2,&2,&2,&3,&2,&2{]}\\
\endarray
$$
Displaying $\mathbf{c}$ over the diagonal and $\mathbf{q}$ on the last column one obtains
$$Q=
\left(
\begin{matrix}
1 & 2 & \alpha_1 & \alpha_2 & \alpha_3 & 14\\
0 & 1 & 2 & \alpha_4 & \alpha_5 &11\\
0 & 0 & 1 & 2 & \alpha_6 & 8\\
0 & 0 & 0 & 1 & 3 & 5\\
0 & 0 & 0 & 0 & 1 & 2\\
0 & 0 & 0 & 0 & 0 & 1
\end{matrix}
\right)
$$
In order to obtain $\alpha_6$ one can use $\left| \begin{matrix} \alpha_6 & 8 \\ 3 & 5 \end{matrix}\right|=1$, which implies
$\alpha_6=5$ and so on, to obtain $\alpha_4=3$, $\alpha_1=3$, $\alpha_5=7$, $\alpha_2=4$, and $\alpha_3=9$.
One can compute for instance $\mu_X(10)$, where the \greedy $X$-decomposition of $10$ is $[0,1,0,1,0,0]$, and $\bar k=6$.
$$
\mu_X(10)=
\frac{(14-1)+(10-1)(6-1)}{14}-
\left(
\begin{smallmatrix}
0&0 & 1 & 0 & 1 & 0
\end{smallmatrix}
\right)
\left(
\begin{smallmatrix}
1 & 2 & 3 & 4 & 9 & 14\\
0 & 1 & 2 & 3 & 7 &11\\
0 & 0 & 1 & 2 & 5 & 8\\
0 & 0 & 0 & 1 & 3 & 5\\
0 & 0 & 0 & 0 & 1 & 2\\
0 & 0 & 0 & 0 & 0 & 1
\end{smallmatrix}
\right)
\left(
\begin{smallmatrix}
0\\1\\0\\1\\0\\0
\end{smallmatrix}
\right)=\frac{15}{7}.
$$
Consider now $f=x^{24}+x^{13}y+x^3y^7+xy^{11}+y^{20}\in \cO_X(10)$. The Newton polygons $\LL(10)$ and $N(f)$ 
are depicted in Figure~\ref{fig-newton}.
Since $I_{N(f)}=1$, $\|10\|_1=2$, and $r_{N(f)}=5$, according to~\eqref{eq-mu2} one has
$$
\mu_X(f)=2I_{N(f)}+\|10\|_1+r_{N(f)}-2+\mu_X(10)=\frac{64}{7}.
$$

\begin{center}
\begin{figure}
\includegraphics[scale=.5]{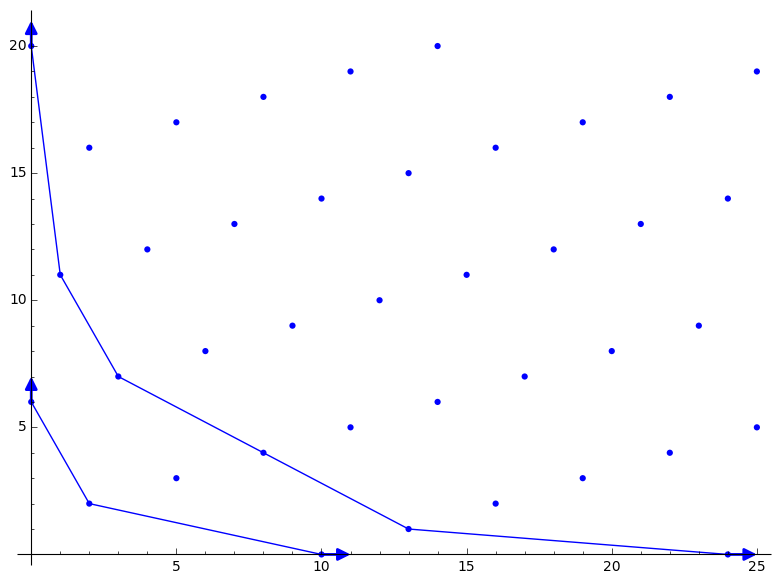}
\put(-250,40){$\LL(10)$}
\put(-210,70){$N(f)$}
\caption{}
\label{fig-newton}
\end{figure}
\end{center}
\end{exam}

\section{LR-logarithmic eigenmodules for generic \texorpdfstring{$\QQ$}{Q}-divisors}
\label{sec-LRlogarithmic}
We begin this section describing the canonical bundle of a cyclic quotient surface singularity $X$.
This description will allow us to give an alternative view on the discrepancy divisor of $X$.
Finally, the $\cO_X$-eigenmodule of LR-logarithmic forms is calculated for generic $\QQ$-divisors.

\subsection{The canonical bundle of \texorpdfstring{$X$}{X}}
Recall that $\mathbf c$ denotes the vector of coefficients of the Hirzebruch-Jung continued fraction of $\frac{d}{q}$
as defined in Section~\ref{sec-arithmetics}. Also recall that the canonical bundle is given by $\cO_X(K_X)=\cO_X(w)$, 
where~$w:=d-1-q$.

\begin{prop}
\label{prop-canonical}
Under the previous notation,
$$\cO_X(w) = \bigotimes_{i=1}^{n} \cO_X(q_i)^{c_i-2}.$$
\end{prop}

\begin{proof}
The result will follow from Theorem~\ref{thm-descomp} after calculating the \greedy $X$-decomposition of 
$w=d-1-q$. Let us denote by $[w]=[w_0=0,w_1,\dots,w_n,w_{n+1}]$ such decomposition. 
Since $d=c_1q_1-q_2$ and $c_1\geq 2$, one has 
$$
w_1= d-1-q_1=(c_1-2)q_1+(q_1-q_2-1)
$$ 
$w_1=c_1-2$. Analogously, by induction assuming $w_i=c_i-2$,
$$
w_{i+1}=
q_i-q_{i+1}-1=(c_{i+1}-2)q_{i+1}+(q_{i+1}-q_{i+2}-1)
$$ 
where $q_i=c_{i+1}q_{i+1}-q_{i+2}$ (note that $c_i\geq 2$), which implies $w_{i+1}=c_{i+1}-2$.
\end{proof}

\subsection{Another view on the discrepancy of cyclic quotient surfaces}
As a consequence of Proposition~\ref{prop-canonical} one can give an alternative proof of the formula for 
the discrepancy divisor of a cyclic quotient singularity, cf.~\cite{Blache-Twoaspects,Tucker-jumping}.
Recall that, given $\pi:\tilde X\to X$ a resolution of the singularity $X$ and 
$E_1,\dots,E_n$ the exceptional divisors of the resolution, the \emph{discrepancy divisor} (or 
\emph{relative canonical divisor}) $K_{\tilde X/X}=\sum_i \varepsilon_i E_i$ associated with 
$\tilde X$ and $\pi$ is given by the formula
\begin{equation}
\label{eq-discrepancy}
K_{\tilde X}=\pi^*(K_X) + K_{\tilde X/X},
\end{equation}
where $K_{\tilde X}$ is the canonical divisor on $\tilde X$ and $K_X=\pi_*(K_{\tilde X})$.

In case $X=X(d;1,q)$ and $\pi$ is the Hirzebruch-Jung resolution described in Section~\ref{sec-quotient},
the canonical divisor $K_X$ is given by a Weil divisor $f\in \cO_X(\w)=\cO_X(\mathbf c -2)$ 
(Proposition~\ref{prop-canonical}), therefore multiplying formula~\eqref{eq-discrepancy} by each $E_j$
one obtains the following vectorial equation
$$
(0,\dots,0)=(c_1-2,\dots,c_n-2) + (\varepsilon_1,\dots,\varepsilon_n) (E_i\cdot E_j),
$$
where $c_j-2=(f) \cdot E_j$ and 
$\sum_i \varepsilon_i (E_i\cdot E_j)=K_{\tilde X/X} \cdot E_j$.
Since
$$E_i\cdot E_j=\begin{cases} -c_i & \text{ if } i=j\\ 1 & \text{ if } |i-j|=1 \\ 0 & |i-j|>1 \end{cases}$$
one obtains $\varepsilon_i=\frac{q_i+\bar q_i}{d}-1$.

\begin{exam}
To continue Example~\ref{ex-X14-11-Mnul} note that the discrepancy divisor associated with $X=X(14;1,11)$ is
$$-\frac{1}{7}\left(E_1+2E_2+3E_3+4E_4+2E_5\right).$$
\end{exam}

\subsection{Proof of Theorem~\ref{thm-Mnul}}
The purpose of this section is to give a description of $\cM^{\nul}_f$ for a generic $f$. Before we give the 
proof, the following technical result is needed.

\begin{lemma}
\label{lemma-kw}
Consider $0\leq k<d$ and $[k]=[0,\dots,0,k_r,\dots,k_s,0\dots,0]$ its \greedy \break $X$-de\-com\-po\-si\-tion, where
$k_r,k_s\neq 0$. Then the \greedy $X$-decomposition of $k+w$ is 
$$[k+w]=[c_{*}-2,c_{r-1}-1,k_r-1,k_{*},k_{s}-1,c_{s+1}-1,c_{*}-2].$$
\end{lemma}

\begin{proof}
We will consider two cases:
\begin{itemize}
 \item Case 1. If $k\geq q+1$, then $k+w\equiv k-q-1 \! \mod d$, with $0\leq k-q-1<d$.
 \item Case 2. If $k<q+1$, in which case $k+w\equiv k+d-q-1 \mod d$, with $0\leq k+d-q-1<d$.
\end{itemize}

Before we prove the result for each case, let us note the following properties:

\begin{enumerate}
 \item\label{prop-1} $\alpha'_{j}:=[0,\dots,0,c_j-1,c_{j+1}-2,\dots,c_{n}-2]$ is a \greedy $X$-decomposition of
 $\|\alpha'_{j}\|_X=q_{j-1}-1$. The proof follows that of Proposition~\ref{prop-canonical}.
 \item\label{prop-2} $\alpha_{i}:=[c_0-2,\dots,c_{i-1}-2,c_i-1,0,\dots,0]$ is a \greedy $X$-decomposition 
 of $\|\alpha'_{i}\|_X=d-q+q_{i+1}$.
 \item\label{prop-4} If $\alpha_1$ and $\alpha_2$ are $X$-decompositions, $\alpha_1$ is \greedy,
 and $\alpha_2\leq \alpha_1$, then $\alpha_2$ is \greedy.
 \item\label{prop-5} If $\alpha=[a_0,\dots,a_i,0,\dots,0]$ (resp.~$\beta=[0,\dots,0,b_{i+1},\dots,b_n]$) is a \greedy 
 $X$-decomposition of $a$ (resp.~$b$) with $b<q_i$, then $\alpha+\beta$ is a \greedy $X$-de\-com\-po\-si\-tion of $a+b$. 
\end{enumerate}

In order to prove case 1, first note that $k>q$ implies $k_1>0$, that is $r=1$ in the statement. Let us define 
$\beta:=[k_0=0,k_1-1,k_2,\dots,k_s-1,0,\dots,0]$ which is the \greedy $X$-decomposition of $\|\beta\|_X=k-q_1-q_s=k-q-q_s$
by~\eqref{prop-4}. Using~\eqref{prop-5} above, one can see that 
$\beta+\alpha_{s+1}=[k_0=0,k_1-1,k_2,\dots,k_{s-1},k_s-1,c_{s+1}-1,c_{s+2}-2,\dots,c_n-2]$ is the 
required \greedy $X$-decomposition of $k-q-1=k+w$.

In order to prove case 2, similarly as before, 
$\beta:=[0,\dots,0,k_r-1,k_{r+1},\dots,k_{s-1}, \break k_s-1,0,\dots,0]$ is the \greedy $X$-decomposition of $\|\beta\|_X=k-q_r-q_s$
by~\eqref{prop-4}. Then using~\eqref{prop-5} $\alpha_{r-1}+\beta+\alpha'_{s+1}$ is the required \greedy $X$-decomposition.
\end{proof}

\begin{proof}[Proof of Theorem~\ref{thm-Mnul}]
The inclusion $\cO_X(k) \otimes \cO_X(w) \subseteq \cM^\nul_f$ is a consequence of $f$ being generic as follows.
Since $\cO_X(k) \otimes \cO_X(w)$ is generated by monomials it is enough to show the result for its monomial generators. 
Consider $h=x^{i_1+i_2}y^{j_1+j_2}$, where $i_1+qj_1=k+m_1d$ and $i_2+qj_2=w+m_2d$ for $m_i\geq 0$. Let us check that 
the pull-back of the $(1,q)$-blow-up of $h\frac{dx\wedge dy}{f}$ can be extended holomorphically to the exceptional divisor. 
Using formula~\eqref{eq-2form} is it enough to check that $N=\frac{\nu_{1,q}(h)-\nu_{1,q}(f)+1+q-d}{d}$ is non-negative, 
that is, $(i_1+i_2)+q(j_1+j_2)-k+1+q-d=(i_1+qj_1-k)+(i_2+qj_2-(d-q-1))=md$, where $m=m_1+m_2\geq 0$.
An induction argument on the number of required blow-ups similar to the one used in Lemma~\ref{lemma-qi} gives the result.

It remains to show that 
$$\dim_\CC \frac{\cO_X(k+w)}{\cM^\nul_f}=\dim_\CC \frac{\cO_X(k+w)}{\cO_X(k) \otimes \cO_X(w)}.$$
The left-hand side dimension is $\K_{X}(f)$ which equals $\sum_{i=1}^n k_i - 1$
by Lemma~\ref{lemma-K-generic}. In other words, one needs to check
$$\dim_\CC \frac{\cO_X(k+w)}{\cO_X(k) \otimes \cO_X(w)}=\sum_{i=1}^n k_i - 1.$$

Let us denote by $\gG_1$ (resp.~$\gG_2$) the Newton polygon associated with $\cO_X(k+w)$ 
(resp.~$\cO_X(k) \otimes \cO_X(w)$). If $S_i\subset \LL(k+w)$ denotes the set of $\LL(k+w)$-points 
in the first quadrant under $\gG_i$, $i=1,2$, then note that $S_i$ are both finite, $S_2\subset S_1$,
and $\#(S_2\setminus S_1)=\dim_\CC \frac{\cO_X(k+w)}{\cO_X(k) \otimes \cO_X(w)}$.
The result will follow from counting $\#(S_2\setminus S_1)=\sum_{i=1}^n k_i - 1$.
Let us write $[k]=[0,\dots,0,k_r,\dots,k_s,0,\dots,0]$, where $k_r,k_s\neq 0$.
By Lemma~\ref{lemma-kw} we know
$$
\ell:=[k+w]=[c_*-2,c_{r-1}-1,k_r-1,k_{*},k_s-1,c_{s+1}-1,c_{*}-2].
$$

By Proposition~\ref{prop-canonical}, $m:=[k]+[w]=[k_*+c_*-2]$.
Given an $X$-decomposition $\alpha=[\alpha_0,\dots,\alpha_{n+1}]$, we will use the following notation:
$$\|\alpha\|_j:=\sum_{i=j}^{n+1}\alpha_iq_i, \quad \quad \text{ and } 
\quad \quad  \|\alpha\|^j:=\sum_{i=0}^{j}\alpha_i\bar q_i.$$

Note that $(\|m\|_j,\|m\|^{j-1})$ denotes the coordinates of a vertex in $\gG_2$ joining the group of segments of 
slope $-\frac{\bar q_{j-1}}{q_{j-1}}$ and the group of segments of slope $-\frac{\bar q_{j}}{q_{j}}$ as shown 
in Figure~\ref{fig-thm-1}. 

\begin{figure}[ht]
\begin{center}
\begin{tikzpicture}[scale=.5]
    \draw [<->,thick] (0,8) node (yaxis) [above] {$ $}
        |- (11,0) node (xaxis) [right] {$ $};
    \draw [thick] (10,0) -- (7,.5);
    \fill (7,.5) circle (2pt);
    \draw [dashed] (7,.5) -- (5,1);
    \fill (5,1) circle (2pt);
    \draw [thick] (5,1) -- (4.5,1.3);
    \draw [dashed] (4.5,1.3) -- (4,1.6);
    \draw [thick] (4,1.6) -- (3.5,1.9);
    \fill (3.5,1.9) circle (2pt) node [above right] {$(\|m\|_{j},\|m\|^{j-1})$};
    \draw [thick] (3.5,1.9) -- (2.5,2.9);
    \draw [dashed] (2.5,2.9) -- (1.5,3.9);
    \draw [thick] (1.5,3.9) -- (.5,5);
    \fill (.5,5) circle (2pt);
    \draw [dashed] (.5,5) -- (0,7);
\end{tikzpicture}
\caption{}
\label{fig-thm-1}
\end{center}
\end{figure}

Let us show following: 
\begin{enumerate}
 \item\label{eq-aux-1} $\|m\|^j=\|\ell\|^j$, $j=0,\dots,r-2$, 
 \item\label{eq-aux-1a} $\|m\|^{r-1}-\|\ell\|^{r-1}=-\bar q_{r-1}$,
 \item\label{eq-aux-2} $\|m\|_j=\|\ell\|_j$, $j=s+2,\dots,n+1$,
 \item\label{eq-aux-2a} $\|m\|_{r}-\|\ell\|_{r}=q_{r-1}$,
 \item\label{eq-aux-3} $\|m\|_j=\|\ell\|_j$, $j=0,\dots,r-1$,
 \item\label{eq-aux-4} $\|m\|^j=\|\ell\|^j$, $j=s+2,\dots,n+1$,
 \end{enumerate}
Equalities~\eqref{eq-aux-1} and~\eqref{eq-aux-2} are immediate since the first $r-1$ (resp.~last $(n-s-1)$) 
coordinates of $\ell$ and $m$ coincide. In order to obtain~\eqref{eq-aux-2a} note that:
$$
\|m\|_{r}-\|\ell\|_{r}=(c_r-1)q_r+\sum_{i=r+1}^{s-1}(c_i-2)q_i+(c_s-1)q_s-q_{s+1}.
$$
Using $c_iq_i=q_{i-1}+q_{i+1}$ one obtains the required formula. Also, in order to obtain~\eqref{eq-aux-3},
$$
\|m\|_{0}-\|\ell\|_{0}=\|m\|_{j}-\|\ell\|_{j}=-q_{r-1}+\|m\|_{r}-\|\ell\|_{r}.
$$
Analogous calculations prove~\eqref{eq-aux-1a} and~\eqref{eq-aux-4}.
Equalities~\eqref{eq-aux-1}-\eqref{eq-aux-4} show that $\gG_1$ and $\gG_2$ share the first $r-1$ groups of 
segments of slopes $-\frac{\bar q_j}{q_j}$, $j=0,\dots,r-2$, and all but one of the segments of slope 
$-\frac{\bar q_{r-1}}{q_{r-1}}$ as shown in Figure~\ref{fig-thm-2}. 

\begin{figure}[ht]
\begin{center}
\begin{tikzpicture}[scale=.5]
    \draw [<->,thick] (0,8) node (yaxis) [above] {$ $}
        |- (11,0) node (xaxis) [right] {$ $};
    \draw [thick] (10,0) -- (7,.5);
    \fill (7,.5) circle (2pt);
    \draw [dashed] (7,.5) -- (6,.75);
    \fill (6,.75) circle (2pt) node [above right] {${(\|m\|_{r},\|m\|^{r-1})=(\|\ell\|_{r}+q_{r-1},\|\ell\|^{r-1}-\bar q_{r-1})}$};
    \draw [thick] (6,.75) -- (4,1.95);
    \fill (4,1.95) circle (2pt);
    \draw [thick] (4,1.95) -- (.3,5.75);
    \fill (.31,5.75) circle (2pt);
    \draw [thick] (6,.75) -- (5,1);
    \fill (5,1) circle (2pt) circle (2pt) node [below left] {${(\|\ell\|_{r},\|\ell\|^{r-1})}$};
    \draw [thick] (5,1) -- (4.5,1.3);
    \draw [dashed] (4.5,1.3) -- (4,1.6);
    \draw [thick] (4,1.6) -- (3.5,1.9);
    \fill (3.5,1.9) circle (2pt);
    \draw [thick] (3.5,1.9) -- (2.5,2.9);
    \draw [dashed] (2.5,2.9) -- (1.5,3.9);
    \draw [thick] (1.5,3.9) -- (.5,5);
    \fill (.5,5) circle (2pt);
    \draw [dashed] (.31,5.75) -- (0,7);
    \draw [thick] (.5,5) -- (.31,5.75);
\end{tikzpicture}
\caption{}
\label{fig-thm-2}
\end{center}
\end{figure}

In particular, the difference $\#(S_2\setminus S_1)$ is invariant if we assume
\begin{equation}
\label{eq-ell}
\ell:=[0,\dots,0,1,k_r-1,k_*,k_s-1,1,0,\dots,0]
\end{equation}
and
$$
m:=[0,\dots,0,0,k_r+c_r-1,k_*+c_*-1,k_s+c_s-1,0,0,\dots,0].
$$
Consider now the quadrilateral $H_r$ given by the vertices $P_r=(\|\ell\|_r,\|\ell\|^{r-1})$, 
$P_{r+1}=(\|\ell\|_{r+1},\|\ell\|^{r})$, $Q_r=(\|m\|_r,\|m\|^{r-1})$, and $Q_{r+1}=(\|m\|_{r+1},\|m\|^{r})$. 
Note that $\overrightarrow{P_rQ_{r}}=(q_r-q_{r+1},\bar q_{r+1}-\bar q_r)$, 
$\overrightarrow{P_rP_{r+1}}=(q_{r+1},-\bar q_{r+1})$ (see Figure~\ref{fig-thm-3}).
\begin{figure}[ht]
\begin{center}
\begin{tikzpicture}[scale=.5]
    \draw [thick,latex-] (10,0) -- (4,1);
    \fill (10,0) circle (2pt) node [right] {$Q_r$};
    \fill (4,1) circle (2pt) node [left] {$P_r$};
    \draw [thick] (10,0) -- (7,6);
    \fill (7,6) circle (2pt);
    \draw [thick,-latex] (4,1) -- (1,7);
    \fill (1,7) circle (2pt) node [left] {$P_{r+1}$};
    \draw [dashed] (1,7) -- (7,6);
    \draw [thick,-latex] (7,6) -- (6,8);
    \draw [thick,-latex] (1,7) -- (6,8);
    \fill (6,8) circle (2pt) node [right] {$Q_{r+1}$};
    \fill (5,4) circle (0pt) node [right] {$A^1_{r}$};
    \fill (4.5,7) circle (0pt) node [right] {$A^2_{r}$};
    \fill (2,4) circle (0pt) node [left] {$(k_r-1)(q_{r},-\bar q_{r})$};
    \fill (8,5) circle (0pt) node [right] {$(k_r+c_r-2)(q_{r},-\bar q_{r})$};
    \fill (6,0) circle (0pt) node [below] {$(q_{r-1}-q_r,\bar q_{r}-\bar q_{r-1})$};
    \fill (3,10) circle (0pt) node [below] {$(q_{r}-q_{r+1},\bar q_{r+1}-\bar q_{r})$};
\end{tikzpicture}
\end{center}
\caption{}
\label{fig-thm-3}
\end{figure}
A simple calculation gives the area of this polygon $H_r$ after decomposing it as a 
parallelogram and a triangle (see Figure~\ref{fig-thm-3}) as $A_r=A_r^1+A_r^2$, where
$$
A_r^1=(k_r-1)(\bar q_r,q_r)\cdot (q_{r-1}-q_r,\bar q_{r}-\bar q_{r-1})=(k_r-1)d
$$
and
$$
\array{c}
A_r^2=
\frac{1}{2}(q_{r}-q_{r+1},\bar q_{r+1}-\bar q_{r})\cdot (c_r-1)(\bar q_{r},q_r)=\\
\\
\frac{1}{2}(c_r-1)(\bar q_{r}q_{r}-\bar q_{r}q_{r+1}+q_r\bar q_{r+1}-\bar q_{r}q_r)=
\frac{1}{2}(c_r-1)d,
\endarray
$$
where Lemma~\ref{lemma-q} is used for these equalities.
Hence $A_r=\frac{1}{2}d(2k_r+c_r-3)$. Using Pick's Theorem for the lattice $\LL(k+w)$ one obtains
$$
\frac{A_r}{d}=\frac{1}{2}B_r+I_r-1=\frac{1}{2}(k_r+k_r+c_r-1)+I_r-1,
$$
where $B_r$ is the number of $\LL(k+w)$-boundary points on the polygon $H_r$, namely 
$\Big((k_r-1)+1\Big)+\Big((k_r+c_r-2)+1\Big)$ and $I_r$ is the number of $\LL(k+w)$-interior points on $H_r$.
Therefore
$$
\frac{1}{2}(2k_r+c_r-3)=\frac{1}{2}(2k_r+c_r-1)+I_r-1
$$
which implies $I_r=0$.
Analogously, one can prove that $I_i=0$, where $I_i$ is the number of $\LL(k+w)$-interior points on $H_i$,
the polygon determined by $P_i$, $P_{i+1}$, $Q_i$, and~$Q_{i+1}$, $i=r+1,\dots,s$.

Finally, this implies that $\#(S_2\setminus S_1)$ can be calculated as the number minus two of boundary 
$\LL(k+w)$-points on the Newton polygon given by $\ell$ in~\eqref{eq-ell}, that is,
$$\gG(\ell)=\gG(q_{r-1})\oplus\gG(q_r)^{k_r-1}\oplus (\oplus_{i=r+1}^{s-1} \gG(q_i)^{k_i})
\oplus \gG(q_s)^{k_s-1}\oplus\gG(q_{s+1}),$$ 
which coincides with $(\|\ell\|_1+1)-2=\sum_i k_i - 1=\|k\|_1-1$ as required.
\end{proof}

\begin{exam}
\label{ex-X14-11-Mnul}
As a continuation of Example~\ref{ex-X14-11-mu} on $X=X(14;1,11)$, we will calculate $\cM^{\nul}(10)$, that is, the 
LR-logarithmic module $\cM^{\nul}_h$ for a generic $h\in \cO_X(10)$. According to
Lemma~\ref{lemma-K-generic} $[k+\w]=[10+2]=[0,1,0,0,0,1,0]$ and $[k]+[\w]=[0,0,1,0,2,0,0]$, Therefore
$\cM^{\nul}(10)=\cO_X(8)\otimes \cO_X(2)^2$. Finally, 
$$\K_X(10)=\dim_\CC \frac{\cO_X(11)\otimes \cO_X(1)}{\cO_X(8)\otimes \cO_X(2)^2}=1.$$ 
\end{exam}

\begin{cor}
The module $\cM^{\nul}_f$ is monomial if $f$ is generic.
\end{cor}

\begin{rem}
In general, $\cM^{\nul}_f$ is not monomial if $f$ is not generic, even if it is a product of curvettes on $X$,
as the following examples shows.
\end{rem}

\begin{exam}
Let $f=(x+y^4)^2-y^{18}\in \O_X(2)$ be a nongeneric germ in $X=X(5;1,4)$. Two consecutive blow-ups of weight
$(4,1)$ and $(1,1)$ respectively serve as a $\Q$-resolution of $(f,0)$. This resolution allows one to use the recursive
formula~\cite[Theorem 4.5]{CMO12}, which results in $\delta_X(f)=\frac{13}{5}$. Let us calculate $\mathcal{M}^{\nul}_{f}$. 
One can easily check that $x^2-y^{8}\in \mathcal{M}^{\nul}_{f}$, since

\begin{eqnarray*}
(x^2-y^{8})\ \dfrac{dx \wedge dy}{(x+y^4)^2-y^{18}}& \leftblowup{x=\bar u_1\bar v_1^4, \ u_1=\bar u_1+1}{y=\bar v_1,\ {v}_1=
\bar v_1^5} & \bar v_1^{8}\ (\bar u_1^2-1)\ \dfrac{d\bar u_1 \wedge d\bar v_1^5}{\bar v_1^8((\bar u_1+1)^2-\bar v_1^{10})}=\\
\\
\\
u_1 (u_1-2)\ \dfrac{du_1 \wedge dv_1}{(u_1^2-v_1^2)}&\leftblowup{u_1=u_2 v_2}{v_1=v_2} &u_2\  
(u_2v_2-2) \dfrac{du_2 \wedge dv_2}{ (u_2-1)(u_2+1)}.\\
\end{eqnarray*}

Analogously one can also check that $y^{13},\ x^2+xy^{4}  \in \mathcal{M}^{\nul}_{f}$. Using
the curvette $h= x^2-y^3\in \cO_X(2)$, one obtains $\K_X(2)=\|2\|_1-1=0$ (Lemma~\ref{lemma-K-generic}).
Also, it is easy to check that $\delta_X(2)=\frac{3}{5}$ (via a $(3,2)$-blow-up as a resolution as mentioned above), 
hence $\R_X(2)=\delta_X(2)-\K_X(2)=\frac{3}{5}$. Since 
$\R_X(2)=\frac{3}{5}=\delta_X(f)-\K_X(f)=\frac{13}{5}-\K_X(f)$, one obtains 
$$
\K_X(f)=2=\dim_{\CC} \frac{\O_X(2)}{{\mathcal M}^{\nul}_{f}}\leq \dim_{\CC} 
\frac{\O_X(2)}{{\CC}\{x^2-y^{8},y^{13},x^2+xy^{4}\}}=2.
$$
Hence, ${\mathcal M}^{\nul}_{f}={\CC}\{x^2-y^{8},y^{13},x^2+xy^{4}\}$,
which is not a monomial module.
\end{exam}

%
%
%
%


\begin{thebibliography}{99}

\bibitem{ABLM-milnor-number}
E.~Artal~Bartolo, J.~Fern{\'a}ndez~de Bobadilla, I.~Luengo, and
A.~Melle-Hern{\'a}ndez,
\newblock \emph{Milnor number of weighted-{L}\^e-{Y}omdin singularities},
\newblock Int. Math. Res. Not. IMRN, (22):4301--4318, 2010.

\bibitem{AMO11a}
E.~{Artal Bartolo}, J.~{Mart{\'i}n-Morales}, and J.~{Ortigas-Galindo},
\newblock \emph{Cartier and {W}eil divisors on varieties with quotient singularities},
\newblock Int. J. Math. 25(11), 2014. 

\bibitem{Blache-Twoaspects}
R.~Blache, \emph{Two aspects of log terminal surface singularities}, Abh. Math.
  Sem. Univ. Hamburg \textbf{64} (1994), 59--87.

\bibitem{Blache-RiemannRoch}
\bysame,
\newblock \emph{Riemann-{R}och theorem for normal surfaces and applications},
\newblock Abh. Math. Sem. Univ. Hamburg, 65:307--340, 1995.

\bibitem{brasselet-milnor}
J.-P. Brasselet, D.~Lehmann, J.~Seade, and T.~Suwa,
\newblock \emph{Milnor classes of local complete intersections},
\newblock Trans. Amer. Math. Soc., 354(4):1351--1371 (electronic), 2002.

\bibitem{MR0450628}
L.~Brenton, \emph{On the {R}iemann-{R}och equation for singular complex
  surfaces}, Pacific J. Math. \textbf{71} (1977), no.~2, 299--312.

\bibitem{JIphd}
J.I. Cogolludo-Agust{\'{\i}}n,
\newblock \emph{Topological invariants of the complement to arrangements of rational
plane curves},
\newblock Mem. Amer. Math. Soc., 159(756):xiv+75, 2002.

\bibitem{CMO12}
J.I. Cogolludo-Agust{\'\i}n, J.~Ortigas-Galindo, and J.~Mart{\'{\i}}n-Morales,
  \emph{Local invariants on quotient singularities and a genus formula for
  weighted plane curves}, Int. Math. Res. Not. (2014), no.~13, 3559--3581.

\bibitem{CMO14}
J.I. Cogolludo-Agust{\'{\i}}n, J.~{Mart{\'i}n-Morales}, and J.~{Ortigas-Galindo},
\newblock \emph{Numerical adjunction formulas for weighted projective planes and counting lattice points},
\newblock Preprint, 2014.

\bibitem{Corti}
A. Corti and J. Koll\'ar, 
\newblock \emph{Existence of canonical flips}, 
\newblock In Flips and abundance for algebraic threefolds. Papers from the Second Summer Seminar on Algebraic 
Geometry held at the University of Utah, Salt Lake City, Utah, August 1991. Ast\'erisque No. 211 (1992), 69--73.

\bibitem{Deligne-intersections}
P. Deligne, \emph{Intersections sur les surfaces regulieres},
In Groupes de Monodromie en G\'eom\'etrie Alg\'ebrique
du Bois-Marie 1967--1969 (SGA 7 II), Lecture
Notes in Math. 340, Springer, Berlin-Heidelberg, 1973.

\bibitem{Dolgachev82}
I.~Dolgachev,
\newblock \emph{Weighted projective varieties},
\newblock In Group actions and vector fields ({V}ancouver, {B}.{C}.,
1981), volume 956 of {\em Lecture Notes in Math.}, pages 34--71. Springer,
Berlin, 1982.

\bibitem{fulton-intersection}
W.~Fulton, \emph{Intersection theory}, second ed., Ergebnisse der Mathematik und
  ihrer Grenzgebiete. 3. Folge. A Series of Modern Surveys in Mathematics
  [Results in Mathematics and Related Areas. 3rd Series. A Series of Modern
  Surveys in Mathematics], vol.~2, Springer-Verlag, Berlin, 1998.

\bibitem{Kouchnirenko-Newton}
A.~G. Kouchnirenko, \emph{Poly\`edres de {N}ewton et nombres de {M}ilnor},
  Invent. Math. \textbf{32} (1976), no.~1, 1--31.

\bibitem{MR0472020}
M.~J. Magazine, G.~L. Nemhauser, and L.~E. Trotter, Jr., \emph{When the greedy
  solution solves a class of knapsack problems}, Operations Res. \textbf{23}
  (1975), no.~2, 207--217.
  

\bibitem{andy}
A.~Mirzaian, 
\newblock \emph{Greedy Coin Change Making, Course Lecture Notes 7}.
\newblock Available at \\ \verb'http://www.cse.yorku.ca/~andy/courses/3101/lecture-notes/CoinChange.pdf' (2015).

\bibitem{Oka-nondegenerate}
M.~Oka, \emph{Non-degenerate complete intersection singularity}, Actualit\'es
Math\'ematiques. [Current Mathematical Topics], Hermann, Paris, 1997.


\bibitem{Ortigas13PhD}
J.~{Ortigas-Galindo},
\newblock \emph{Algebraic and Topological Invariants of Curves and Surfaces with
Quotient Singularities},
\newblock PhD thesis, 2013.
\newblock \verb'http://zaguan.unizar.es/record/11738'.

\bibitem{Ortigas-cr}
\bysame, \emph{Generators of the cohomology algebra of the
  complement to a rational algebraic curve in the weighted projective plane
  {$\mathbb{P}_\omega^2$}}, C. R. Math. Acad. Sci. Paris \textbf{352} (2014),
  no.~1, 65--70.

\bibitem{MR927963}
M.~Reid, \emph{Young person's guide to canonical singularities}, Algebraic
  geometry, {B}owdoin, 1985 ({B}runswick, {M}aine, 1985), Proc. Sympos. Pure
  Math., vol.~46, Amer. Math. Soc., Providence, RI, 1987, pp.~345--414.
  
\bibitem{Reid-Surface}
\bysame,
\newblock \emph{Surface cyclic quotient singularities and Hirzebruch-Jung resolutions}.
\newblock Available at \verb'http://www.maths.warwick.ac.uk/~miles/surf/' (1997).

\bibitem{STV-Milnornumbers}
J.~Seade, M.~Tib{\u{a}}r, and A.~Verjovsky,
\newblock \emph{Milnor numbers and {E}uler obstruction},
\newblock Bull. Braz. Math. Soc. (N.S.), 36(2):275--283, 2005.

\bibitem{Steenbrink77}
J.H.M. Steenbrink,
\newblock \emph{Mixed {H}odge structure on the vanishing cohomology},
\newblock In Real and complex singularities ({P}roc. {N}inth {N}ordic
{S}ummer {S}chool/{NAVF} {S}ympos. {M}ath., {O}slo, 1976), pages 525--563.
Sijthoff and Noordhoff, Alphen aan den Rijn, 1977.

\bibitem{Le-Someremarks}
L{\^e}~D. Tr\'ang,
\newblock \emph{Some remarks on relative monodromy},
\newblock In Real and complex singularities ({P}roc. {N}inth {N}ordic
{S}ummer {S}chool/{NAVF} {S}ympos. {M}ath., {O}slo, 1976), pages 397--403.
Sijthoff and Noordhoff, Alphen aan den Rijn, 1977.

\bibitem{Tucker-jumping}
K.~Tucker, \emph{Jumping numbers on algebraic surfaces with rational
  singularities}, Trans. Amer. Math. Soc. \textbf{362} (2010), no.~6,
  3223--3241.

\end{thebibliography}
\end{document}